\documentclass[11pt,a4paper,leqno]{amsart}

\usepackage[latin1]{inputenc}
\usepackage[T1]{fontenc}
\usepackage{amsfonts}
\usepackage{amsmath}
\usepackage{amssymb}
\usepackage{eurosym}
\usepackage{mathrsfs}
\usepackage{palatino}
\usepackage{color}
\usepackage{esint}
\usepackage{url}
\usepackage{verbatim}

\usepackage{enumerate}

\newcommand{\R}{\mathbb{R}}
\newcommand{\C}{\mathbb{C}}

\newcommand{\N}{\mathbb{N}}

\newcommand{\Z}{\mathbb{Z}}
\newcommand{\E}{\mathbb{E}}

\newcommand{\calM}{\mathcal{M}}
\newcommand{\calS}{\mathcal{S}}

\newcommand{\calR}{\mathcal{R}}

\newcommand{\bla}{\big \langle}
\newcommand{\bra}{\big \rangle}

\numberwithin{equation}{section}

\newcommand{\ud}[0]{\,\mathrm{d}}

\newcommand{\dist}[0]{\operatorname{dist}}

\newcommand{\abs}[1]{|#1|}
\newcommand{\Babs}[1]{\Big|#1\Big|}

\newcommand{\Norm}[2]{\|#1\|_{#2}}

\newcommand{\BNorm}[2]{\Big\|#1\Big\|_{#2}}
\newcommand{\pair}[2]{\langle #1,#2 \rangle}
\newcommand{\Bpair}[2]{\Big\langle #1,#2 \Big\rangle}
\newcommand{\ave}[1]{\langle #1\rangle}

\newcommand{\Off}[0]{\operatorname{Off}}

\newcommand{\BMO}[0]{\operatorname{BMO}}
\newcommand{\supp}[0]{\operatorname{spt}}
\newcommand{\loc}[0]{\operatorname{loc}}

\newcommand{\osc}[0]{\operatorname{osc}}

\newcommand{\D}[0]{\mathbb{D}}
\newcommand{\eps}[0]{\varepsilon}

\newcommand{\UMD}{\operatorname{UMD}}

\newcommand{\ch}[0]{\operatorname{ch}}

\newcommand{\calD}[0]{\mathcal{D}}

\newcommand{\wt}[1]{{\widetilde{#1}}}

\swapnumbers
\theoremstyle{plain}
\newtheorem{thm}[equation]{Theorem}
\newtheorem{lem}[equation]{Lemma}
\newtheorem{prop}[equation]{Proposition}

\theoremstyle{definition}
\newtheorem{defn}[equation]{Definition}

\theoremstyle{remark}
\newtheorem{rem}[equation]{Remark}

\title{Off-diagonal estimates for bi-commutators}

\author{Emil Airta}
\author{Tuomas Hyt\"onen}
\author{Kangwei Li}
\author{Henri Martikainen}
\author{Tuomas Oikari}

\address[E.A., T.H., H.M, \& T.O.]{Department of Mathematics and Statistics, University of Helsinki, P.O.B. 68, FI-00014 University of Helsinki, Finland}
\email{emil.airta@helsinki.fi}
\email{tuomas.hytonen@helsinki.fi}
\email{henri.martikainen@helsinki.fi}
\email{tuomas.v.oikari@helsinki.fi}

\address[K.L.]{Center for Applied Mathematics, Tianjin University, Weijin Road 92, 300072 Tianjin, China}
\email{kli@tju.edu.cn}

\makeatletter
\@namedef{subjclassname@2010}{%
  \textup{2010} Mathematics Subject Classification}
\makeatother


\begin{document}

\allowdisplaybreaks

\begin{abstract}
We study the bi-commutators $[T_1, [b, T_2]]$ of pointwise multiplication and Calder\'on--Zygmund operators, and characterize their $L^{p_1}L^{p_2} \to L^{q_1}L^{q_2}$ boundedness for several off-diagonal regimes of the mixed-norm integrability exponents $(p_1,p_2)\neq(q_1,q_2)$. The strategy is based on a bi-parameter version of the recent approximate weak factorization method.
\end{abstract}

\maketitle

\section{Introduction}
In this paper we characterise many new estimates for \textbf{bi}-commutators. The classical commutators have the form $[b,T]\colon f \mapsto bTf - T(bf)$, where $T$ is a singular integral operator
\begin{equation}\label{eq:SIO}
Tf(x) = \int_{\R^d} K(x,y)f(y)\ud y.
\end{equation}
The expression \eqref{eq:SIO} represents a broad class of linear transformations of functions arising across analysis.
Distinguished special cases are the Hilbert transform $H$ in dimension $d = 1$, which has the kernel $K(x,y) = \frac{1}{x-y}$, and the Riesz transforms $R_j$ in dimensions $d \ge 2$, which have
the kernel $K_j(x,y) = \frac{x_j-y_j}{|x-y|^{d+1}}$, $j=1,\ldots, d$.

The Hilbert transform lies in the scope of complex analysis, and the Hilbert commutator $[b,H]$ is connected to Hankel operators.
A classical theorem of Nehari \cite{Ne} characterises the boundedness of the Hilbert commutators $[b,H]$ via this link. On the other hand,
the result of Coifman--Rochberg--Weiss \cite{CRW} brings the commutators to the heart of harmonic analysis by showing that
\begin{equation}\label{eq:commutatorbb}
\|b\|_{\BMO} \lesssim \|[b,T]\|_{L^p(\R^d) \to L^p(\R^d)} \lesssim \|b\|_{\BMO}, \qquad p \in (1,\infty),
\end{equation}
for a class of non-degenerate singular integrals $T$ on $\R^d$. Here $\BMO$ stands for the usual space of functions of bounded mean oscillation:
$$
\Norm{b}{\BMO} :=\sup_I \fint_I  \abs{b-\ave{b}_I},
$$
where the supremum is over all cubes $I \subset \R^d$ and $\ave{b}_I = \fint_I b := \frac{1}{|I|} \int_I b$.

The off-diagonal situation $[b,T]\colon L^p\to L^q$, $p \ne q$, is also completely understood. In the case $1<p<q<\infty$, a two-sided estimate like \eqref{eq:commutatorbb}
with $\BMO$ replaced by the homogeneous H\"older space $\dot C^{0,\alpha}$,
$$
\Norm{b}{\dot C^{0,\alpha}} := \sup_{x \ne y} \frac{|b(x) - b(y)|}{|x-y|^{\alpha}}, \qquad \alpha :=d\Big(\frac{1}{p}-\frac{1}{q}\Big),
$$
was obtained by Janson \cite{Jan}. The remaining range with $1<q<p<\infty$ was characterised only very recently in \cite{Hy5}:
$$
 \|[b,T]\|_{L^p \to L^q} \sim \Norm{b}{\dot L^{r}} := \inf_{c \in \C} \Norm{b-c}{L^r}, \qquad \frac{1}{r} :=\frac{1}{q}-\frac{1}{p}.
$$

As the commutator annihilates constants, all of the spaces above also have the feature that they do not see constants. This philosophy is more complicated
in the bi-parameter setting where the bi-commutator annihilates all functions that depend only on one of the variables $(x_1, x_2) \in \R^{d_1} \times \R^{d_2}$.

The motivation for commutator estimates stems from their many applications and connections to modern harmonic analysis. For example, the $L^p\to L^p$ characterization yields
factorizations for Hardy spaces \cite{CRW} and implies various div-curl lemmas relevant for compensated compactness \cite{CLMS}. In \cite{Hy5}, the off-diagonal $L^p\to L^q$ boundedness for $q<p$ is connected to a conjecture of Iwaniec \cite{Iwa} about the prescribed Jacobian problem.

In this paper, we address the question of the off-diagonal $L^p(\R^d) \to L^q(\R^d)$ boundedness of the bi-commutators $[T_1, [b, T_2]]$, where each $T_i$
is a Calder\'on--Zygmund operator on $\R^{d_i}$ and $\R^d$ is viewed as the bi-parameter product space $\R^d = \R^{d_1} \times \R^{d_2}$.
Due to the product space nature of the problem, it is natural to introduce an additional level of generality by allowing different integrability exponents in the $x_1$ and $x_2$ coordinates, thereby leading to the question of $L^{p_1}_{x_1}L^{p_2}_{x_2}\to L^{q_1}_{x_1}L^{q_2}_{x_2}$ boundedness. In contrast to the three qualitatively different regimes $p<q$, $p=q$ and $p>q$ for the commutator, there will now be nine different situations depending on the relative size of both $p_1,q_1$ on the one hand, and $p_2,q_2$ on the other hand.

We are interested specifically on the off-diagonal case $(p_1, p_2) \ne (q_1, q_2)$. The diagonal, involving the product $\BMO$ space of Chang and Fefferman \cite{CF1, CF2},
has e.g. been studied in Ferguson--Sadosky \cite{FS} and Ferguson--Lacey \cite{FL}, with further extensions appearing in \cite{LPPW,LPPW2,OPS}.

Our question is the natural bi-parameter analogue of \cite{Hy5}. In fact, \cite{Hy5} already suggests a naive conjecture, where in each regime of the exponents,
we should have the corresponding natural vector-valued space. For example, in the case
$$
p_1 < q_1 \textup{ and } p_2 = q_2
$$
we could expect the space $\dot C^{0, \beta_1}_{x_1}(\BMO_{x_2})$
and in the case
$$
p_1 = q_1 \textup{ and } p_2 < q_2
$$
we could expect the space $\BMO_{x_1}(\dot C^{0, \beta_2}_{x_2})$, where $\beta_i :=d_i\Big(\frac{1}{p_i}-\frac{1}{q_i}\Big)$. Somewhat strikingly,
we show that the naive conjecture is not always the right one. In the second case, the space $\dot C^{0, \beta_2}_{x_2}(\BMO_{x_1})$,
where the order of the spaces $\dot C^{0, \beta_2}_{x_2}$ and $\BMO_{x_1}$ has been switched to an unexpected order, provides the correct
sufficient and necessary condition. 

The following Theorem \ref{thm:main} is our main result. It verifies the naive conjecture in some cases and proves an unexpected necessary and sufficient condition in some other cases.
Some characterisations are left for future work as they do not appear to be amenable to our current methods.
The stated theorem is a simplified and shortened version of the obtained new estimates -- for example, the symmetry assumption is used here purely for convenience in order not to have to include partial adjoints $[T_1^*, [b,T_2]]$ in the estimates. In the three lower-right cases of the table of Theorem \ref{thm:main}, where only an upper bound is stated, we also obtain certain related lower bounds, but these do not admit a simple formulation in terms of classical function spaces, and are therefore omitted in this Introduction.
\begin{thm}\label{thm:main}
Let $T_1$ and $T_2$ be two symmetrically non-degenerate CZOs on $\R^{d_1}$ and $\R^{d_2}$, respectively, $b \in L^2_{\loc}(\R^{d_1+d_2})$ and $p_1, p_2, q_1, q_2 \in (1,\infty)$. Let
\begin{equation*}
\begin{split}
  \beta_i :=d_i\Big(\frac{1}{p_i}-\frac{1}{q_i}\Big),\quad \text{if}\quad p_i<q_i;\qquad
  \frac{1}{r_i} :=\frac{1}{q_i}-\frac{1}{p_i},\quad \text{if}\quad p_i>q_i.
\end{split}
\end{equation*}
Then $\|[T_1, [b, T_2]]\|_{L^{p_1}_{x_1}L^{p_2}_{x_2}\to L^{q_1}_{x_1}L^{q_2}_{x_2}}$ has upper and lower bounds according to the following table:
\begin{center}
\begin{tabular}{ c | c | c | c }
		& $p_1<q_1$ & $p_1=q_1$ & $p_1>q_1$ \\ \hline
		& & & \\
$p_2<q_2$ &  $\sim\| b\|_{\dot C^{0,\beta_1}_{x_1}(\dot C^{0,\beta_2}_{x_2})}$	&  $\sim\|b\|_{\dot C^{0,\beta_2}_{x_2}(\BMO_{x_1})}$	
& $\lesssim\|b\|_{\dot L^{r_1}_{x_1}(\dot C^{0,\beta_2}_{x_2})}$ \\
& & &  $\gtrsim\|b\|_{\dot C^{0,\beta_2}_{x_2}(\dot L^{r_1}_{x_1})}$ \\ \hline
& &  & \\
$p_2=q_2$ & $\sim\|b\|_{\dot C^{0,\beta_1}_{x_1}(\BMO_{x_2})}$	& $\lesssim \|b\|_{\BMO_{\rm prod}}$  &  $\lesssim\|b\|_{\dot L^{r_1}_{x_1}(\BMO_{x_2})}$ \\
	& & $\gtrsim \|b\|_{\BMO_{\rm rect}}$ & \\ \hline
	& & & \\
$p_2>q_2$ & $\sim\| b\|_{\dot C^{0,\beta_1}_{x_1}(\dot L^{r_2}_{x_2})}$	& $\lesssim\| b\|_{\BMO_{x_1}(\dot L^{r_2}_{x_2})}$ & $\lesssim\| b\|_{\dot L^{r_1}_{x_1}(\dot L^{r_2}_{x_2})}$	\\
\end{tabular}
\end{center}
\end{thm}

It is important to realize that commutator upper bounds are expected for all bounded singular integrals, while the lower bounds obviously require
some non-degeneracy. All our lower bounds hold under weak non-degeneracy assumptions on the operators, and neither the homogeneity nor the translation-invariance of the operators is needed. However, they have been crucial in the diagonal case \cite{FL, LPPW}.

There has been a vast amount of other recent activity regarding both one-parameter and multi-parameter commutators.
For example, the two-weight commutator estimates, which include
\cite{HLW, HLW2, HPW, Hy5, LOR1, LOR2, LMV:Bloom, LMV:BloomProdBMO}, have been one of the main lines of development.
Commutators are also actively studied in other settings: see for example \cite{DLOPW} for the flag setting, \cite{DLOPW2} for the Zygmund dilation setting and \cite{DLOWY} for the Bessel setting.

The motivation for this paper is not just the results in Theorem \ref{thm:main} but also the associated methodology. We develop some interesting bi-parameter versions of the flexible one-parameter commutator methods \cite{Hy5}.
In more detail, our strategy requires us to develop new bi-parameter methods that exploit the following known interplay between commutator bounds and weak factorizations.
The logic of this connection is that while Nehari's theorem \cite{Ne} may be seen as a corollary of the classical factorization of $H^1(\D)$, Coifman, Rochberg and Weiss \cite{CRW} reversed the reasoning and deduced from the commutator estimate \eqref{eq:commutatorbb}, using duality, a \emph{weak} factorization (involving sums of products, rather than just products) of the real-variable Hardy space $H^1(\R^d)$.
This connection between commutators estimates and factorizations of functions stems from identities like
\begin{equation}\label{eq:connection}
\langle [b,T]f, g\rangle = \langle b, Tf \cdot g - f \cdot T^* g\rangle.
\end{equation}
Recently, in \cite{Hy5} the idea was to reverse the reasoning again, and to first directly prove a suitable factorization, and then
use it to prove the desired commutator estimate. This leads to the approximate weak factorization method, where a function $h$ is expanded
in the form
\begin{equation}\label{eq:oneparWF}
h = \frac{h}{T^* g} T^* g =: -f \cdot T^*g = -f \cdot T^*g + Tf \cdot g + h', \qquad h' := -Tf \cdot g,
\end{equation}
and a suitable $g$ above allows to absorb the error $\|h'\| \ll \|h\|$.

Instead of \eqref{eq:connection}, in the bi-parameter setting we have
\begin{equation*}
  \langle [T_1, [b,T_2]]f, g\rangle = \langle b, T_2 f \cdot T_1^* g - f \cdot T_1^* T_2^* g - T_1 T_2 f \cdot g + T_1 f \cdot T_2^*g\rangle.
\end{equation*}
We develop our new weak factorizations in Section \ref{sec:WF}.
In Lemma \ref{lem:IteratedFactorization} we expand a function $f$ supported on a rectangle $R = I \times J$ and satisfying $\int_I f = \int_J f = 0$ using
the approximate weak factorization logic. This produces more error terms, with more complicated supports, when compared to the one-parameter analogue \eqref{eq:oneparWF}, and makes
the repeated use of the factorization more tricky. Nevertheless, we manage to use our bi-parameter approximate weak factorization
to e.g. prove Theorem \ref{thm:mixednormrec}, which is the key to many commutator lower bounds.

Here is an outline of the paper. After the short preliminaries of Section \ref{sec:prelim} we move on to the oscillatory characterizations of
various function spaces in Section \ref{sec:OscCharofSpaces}. These function space characterizations will be combined
with the weak factorizations of Section \ref{sec:WF} to prove our necessary conditions for commutator boundedness in Section \ref{sec:nes}.
Section \ref{sec:suf} collects all the sufficient conditions. Sections \ref{sec:nes} and \ref{sec:suf} combined give Theorem \ref{thm:main}.
Section \ref{sec:vecval} records some additional vector-valued estimates for commutators that are of independent interest -- see e.g. Theorem \ref{thm:vecval2par}.



\subsection*{Acknowledgements}
E. Airta, T. Hyt\"onen, H. Martikainen and T. Oikari  were supported by the Academy of Finland through project Nos. 294840 (Martikainen), 327271 (Airta, Martikainen), 306901 (Airta, Martikainen, Oikari), and 314829 (Hyt\"onen), and through the  Finnish Centre of Excellence in Analysis and Dynamics Research (project No. 307333, all four). Martikainen and Oikari were also supported by the three-year research grant 75160010 of the University of Helsinki.
\section{Preliminaries}\label{sec:prelim}

\subsection{Basic notation}
We denote $A \lesssim B$ if $A \le CB$ for some absolute constant $C$. The constant $C$ can at least depend on the dimensions of the appearing Euclidean spaces, on integration exponents, and on various Banach space constants. It can also depend on various other fixed constants, like those related to singular integrals, and so on.
We denote $A \sim B$ if $B \lesssim A \lesssim B$.

When we consider $\R^d$ as a bi-parameter product space $\R^d = \R^{d_1} \times \R^{d_2}$ we often denote the mixed-norm space $L^{p_1}(\R^{d_1}; L^{p_2}(\R^{d_2}))$
by $L^{p_1}_{x_1} L^{p_2}_{x_2}$ -- this is suggested by the notation $x = (x_1, x_2) \in  \R^d = \R^{d_1} \times \R^{d_2}$.
We also always identify $f \colon \R^d \to \C$ satisfying
$$
\Big(\int_{\R^{d_1}} \Big( \int_{\R^{d_2}} |f(x_1, x_2)|^{p_2} \ud x_2 \Big)^{p_1/p_2} \ud x_1\Big)^{1/p_1} < \infty
$$
with the function $\phi_f \in L^{p_1}(\R^{d_1}; L^{p_2}(\R^{d_2}))$, $\phi_f(x_1) = f(x_1, \cdot)$.

Often integral pairings -- denoted with the bracked notation $\langle f, g \rangle = \int fg$ -- need to be taken with respect to one of the variables only.
For example, if $f \colon \R^{d} \to \C$ and $h \colon \R^{d_1} \to \C$, then $\langle f, h \rangle_1 \colon \R^{d_2} \to \C$ is defined by
$$
\langle f, h \rangle_1(x_2) = \int_{\R^{d_1}} f(y_1, x_2)h(y_1)\ud y_1.
$$
We denote averages by
$$
\langle f \rangle_A = \fint_A f := \frac{1}{|A|} \int_A f,
$$
where $|A|$ is the Lebesgue measure of the set $A$. The indicatator function of a set $A$ is denoted by $1_A$.
We try to denote cubes in $\R^{d_i}$ by $I_i, J_i, L_i$ and so on --  that is, the dimension of the cube can be read from the subscript. Various rectangles then take the form $I_1 \times I_2$, $J_1 \times J_2$, etc. The side length of a cube $I_i$ is denoted by $\ell(I_i)$.

\subsection{Singular integrals and commutators}
We call $$K_i \colon \R^{d_i} \times \R^{d_i} \setminus \{(x_i,y_i) \in \R^{d_i} \times \R^{d_i} \colon x_i =y_i \} \to \C$$ a standard Calder\'on-Zygmund kernel on $\R^{d_i}$
if we have
$$
|K(x_i,y_i)| \le \frac{C}{|x_i-y_i|^{d_i}}
$$
and, for some $\alpha_i \in (0,1]$, we have
\begin{equation}\label{eq:holcon}
|K(x_i, y_i) - K(x_i', y_i)| + |K(y_i, x_i) - K(y_i, x_i')| \le C\frac{|x_i-x_i'|^{\alpha}}{|x_i-y_i|^{d_i+\alpha_i}}
\end{equation}
whenever $|x_i-x_i'| \le |x_i-y_i|/2$. 

Many of our results hold with \eqref{eq:holcon} replaced by a significantly weaker assumption -- see Remark
\ref{rem:modcon}. However, we do not emphasise this too much as it is not a novelty of this paper (see e.g. \cite{Hy5}).

A singular integral operator (SIO) is a linear operator $T_i$ on $\R^{d_i}$ (initially defined, for example, on bounded and compactly supported functions) so that there is a standard kernel
$K_i$ for which
$$
\langle T_if ,g \rangle = \iint_{\R^{d_i} \times \R^{d_i}} K_i(x_i,y_i) f(y_i) g(x_i) \ud y_i \ud x_i
$$
whenever the functions $f$ and $g$ are nice and have disjoint supports. A Calder\'on--Zygmund operator (CZO) is an SIO $T_i$, which is bounded
from $L^p(\R^{d_i}) \to L^p(\R^{d_i})$ for all (equivalently for some) $p \in (1,\infty)$. The $T1$ theorem \cite{DJ} says that an SIO is a CZO if and only if
$$
\int_{I_i} |T_i1_{I_i}| + \int_{I_i} |T_i^*1_{I_i}| \lesssim |I_i|
$$
for all cubes $I_i \subset \R^{d_i}$. Here $T_i^*$ is the linear adjoint of $T_i$. We know a lot about the structure of a CZO $T_i$:
we can represent $T_i$ with certain dyadic model operators (DMOs) -- see \cite{Hy, Hy2}.
We will have use for this later.

If $b \in L^s_{\loc}(\R^d)$ for some $s \in (1,\infty)$ and $T$ is a CZO on $\R^d$, then the pairing
$$
\langle [b,T]f, g\rangle = \langle Tf, bg\rangle - \langle T(bf), g\rangle
$$
is well-defined for $f, g \in L^{\infty}_c$. If $\R^d = \R^{d_1} \times \R^{d_2}$ and $T_i$ is a CZO on $\R^{d_i}$, we can similarly
define $\langle [T_1, [b,T_2]]f, g\rangle$. Then we can ask, via duality, if $[T_1, [b,T_2]]$ maps $L^{p_1}_{x_1} L^{p_2}_{x_2} \to L^{q_1}_{x_1} L^{q_2}_{x_2}$.
However, when a commutator lower bound is proved, the full norm $$\|[T_1, [b,T_2]]\|_{L^{p_1}_{x_1} L^{p_2}_{x_2} \to L^{q_1}_{x_1} L^{q_2}_{x_2}}$$ is not actually needed. We will define so-called off-support versions of the norm, which can be defined even if we only have $b \in L^1_{\loc}$.
In fact, these off-support constants depend only on the kernels $K_1$ and $K_2$ and not on the CZOs themselves.

When we are given a CZO $T_i$ we always write $K_i$ for its kernel without explicit mention.

\section{Function spaces and oscillatory characterizations}\label{sec:OscCharofSpaces}

In our goal of linking the boundedness of the commutator $[T_1,[b,T_2]]$ with the membership of $b$ in a suitable function space, a useful intermediate notion is provided by various oscillatory characterizations of the different function space norms. In this section we specifically go through all the oscillatory conditions that appear in our
commutator \emph{lower} bounds.

\subsection{One-parameter spaces}
We begin by recalling the relevant characterizations in the one-parameter situation, as this will motivate the necessarily more complicated expressions in the two parameter case.

The case of $\BMO$ is most immediate, as the norm is directly given by the oscillatory quantity
\begin{equation*}
  \Norm{b}{\BMO(\R^d)}:=\sup_I\fint_I\abs{b-\ave{b}_I}.
\end{equation*}
For the homogeneous H\"older norms
$$
\Norm{b}{\dot C^{0,\alpha}(\R^d)} := \sup_{x \ne y} \frac{|b(x) - b(y)|}{|x-y|^{\alpha}},
$$
there is a well-known analogous equivalent norm:
\begin{prop}\label{prop:CalphaOsc}
We have
\begin{equation*}
  \Norm{b}{\dot C^{0,\alpha}(\R^d)}\sim\sup_I \frac{1}{\ell(I)^\alpha} \fint_I\abs{b-\ave{b}_I},
\end{equation*}
where the supremum is over all cubes $I \subset \R^d$.
\end{prop}

\begin{proof}
``$\gtrsim$'' is immediate, since $\abs{b(x)-\ave{b}_I}\leq\fint_I\abs{b(x)-b(y)}\ud y$, where $\abs{b(x)-b(y)}\lesssim\Norm{b}{\dot C^{0,\alpha}(\R^d)}\ell(I)^\alpha$ for all $x,y\in I$.

For ``$\lesssim$'', denote the right-hand side of the claim by $N$, and fix $x\neq y\in\R^d$. Define $x_k:=(1-2^{-k})x+2^{-k}y$ and $y_k:=(1-2^{-k})y+2^{-k}x$, and note that $x_1=y_1=\frac12(x+y)$.
If $Q_k(x):=Q(x_k,2^{-k}\abs{x-y})$ is the cube of centre $x_k$ and side-length $2\cdot 2^{-k}\abs{x-y}$, we easily check that $Q_{k+1}(x)\subset Q_k(x)$ and
$Q_k(x) \subset B(x, C2^{-k}|x-y|)$.
Thus, we have
\begin{equation*}
  b(x)
  =\sum_{k=1}^\infty\Big(\fint_{Q_{k+1}(x)}b-\fint_{Q_k(x)}b\Big)+\fint_{Q(\frac12(x+y),\frac12\abs{x-y})}b
\end{equation*}
and hence
\begin{equation*}
\begin{split}
  \Babs{b(x)-\fint_{Q(\frac12(x+y),\frac12\abs{x-y})}b}
  &\leq\sum_{k=1}^\infty\fint_{Q_{k+1}(x)}\Babs{b-\ave{b}_{Q_k(x)}}
    \lesssim\sum_{k=1}^\infty\fint_{Q_{k}(x)}\Babs{b-\ave{b}_{Q_k(x)}} \\
  &\lesssim N\sum_{k=1}^\infty\ell(Q_k(x))^\alpha
  \lesssim N\abs{x-y}^\alpha\sum_{k=1}^\infty 2^{-k\alpha}
  \lesssim N\abs{x-y}^\alpha.
\end{split}
\end{equation*}
A similar bound for $y$ in place of $x$ and the triangle inequality show that $\abs{b(x)-b(y)}\lesssim N\abs{x-y}^\alpha$.
\end{proof}

The final space of interest in the one-parameter case is
$$
\Norm{b}{\dot L^{r}(\R^d)} := \inf_c \Norm{b-c}{L^r},
$$
where the infimum is taken over all constants.
For this space we have the following characterisation -- similar estimates already appeared in \cite{Hy5}, but we single them out here
as a separate proposition. A collection of cubes $\mathscr S$ is called $\gamma$-sparse if there are
pairwise disjoint subsets $E(S)\subset S$, $S \in \mathscr S$, with $\abs{E(S)} \ge \gamma\abs{S}$. We can often work with $\gamma = 1/2$ and simply talk about sparseness.
\begin{prop}\label{prop:LrOsc}
For $r \in (1,\infty)$ we have
\begin{equation*}
\begin{split}
  \Norm{b}{\dot L^{r}(\R^d)} &\sim\sup_Q\Norm{b-\ave{b}_Q}{L^r(Q)}\sim\Norm{M^{\#}b}{L^r(\R^d)} \\
  &\sim \sup\Big\{\sum_{S\in\mathscr S}\lambda_S\int_S\abs{b-\ave{b}_S}:\mathscr S \text{ is sparse},\sum_{S\in\mathscr S}\abs{S}\lambda_S^{r'}\leq 1\Big\},
\end{split}
\end{equation*}
where $M^{\#}b$ is the sharp maximal function $M^{\#}b(x) = \sup_{Q \ni x} \fint_{Q} |b - \ave{b}_Q|$ and the supremum is taken over
all cubes $Q \subset \R^d$.
\end{prop}

\begin{proof}
We prove a chain of upper bounds both starting and finishing with $\Norm{b}{\dot L^r(\R^d)}$, and covering all other expressions as intermediate steps.
To begin with, we have
\begin{equation*}
  \Norm{b}{\dot L^{r}(\R^d)}\lesssim\sup_Q\Norm{b-\ave{b}_Q}{L^r(Q)} =: N.
\end{equation*}
To see this, fix an increasing sequence of cubes $(Q_i)_i$ that exhaust $\R^d$. Let $i \le j$ and write
\begin{align*}
|\ave{b}_{Q_j} - \ave{b}_{Q_i}| &= \Big( \fint_{Q_i} | [\ave{b}_{Q_j} - b(x)] + [b(x) - \ave{b}_{Q_i}] |^r \ud x \Big)^{1/r} \\
&\le |Q_i|^{-1/r} ( \| b-\ave{b}_{Q_j}\|_{L^r(Q_i)} + \| b-\ave{b}_{Q_i}\|_{L^r(Q_i)}) \\
&\le |Q_i|^{-1/r}  ( \| b-\ave{b}_{Q_j}\|_{L^r(Q_j)} + \| b-\ave{b}_{Q_i}\|_{L^r(Q_i)})
\le 2|Q_i|^{-1/r} N.
\end{align*}
Thus, $(\ave{b}_{Q_j})_j$ is a Cauchy sequence and the limit $c := \lim_{j \to \infty} \ave{b}_{Q_j}$ exists. Fatou's lemma yields the desired estimate:
$$
\Norm{b}{\dot L^{r}(\R^d)}^r \le \int_{\R^d} |b-c|^r = \int \lim_{j \to \infty} 1_{Q_j} |b-\ave{b}_{Q_j}|^r
\le \liminf_{j \to \infty} \int_{Q_j} |b-\ave{b}_{Q_j}|^r \le N^r.
$$

For a fixed cube $Q \subset \R^d$, we have
\begin{equation*}
  1_Q\abs{b-\ave{b}_Q}\lesssim\sum_{S\in\mathscr S}1_S\fint_S\abs{b-\ave{b}_S}
\end{equation*}
for a suitable sparse subcollection $\mathscr S\subset\mathscr D(Q)$. Here $\mathscr D(Q)$ consists of the dyadic subcubes of $Q$ obtained by dividing $Q$ in the natural way.
 For this elementary variant of Lerner's oscillation formula see e.g. \cite[Lemma 3.4]{Hy5}. Using this we now have
\begin{equation*}
  \Norm{b-\ave{b}_Q}{L^r(Q)}
  \lesssim\BNorm{\sum_{S\in\mathscr S}1_S\fint_S\abs{b-\ave{b}_S}}{L^r(\R^d)}
  \lesssim\Big(\sum_{S\in\mathscr S}\abs{S}\Big[\fint_S\abs{b-\ave{b}_S}\Big]^r\Big)^{1/r},
\end{equation*}
where the last step is easily verified by dualising with $\phi\in L^{r'}$ and using the definition of sparseness:
\begin{equation}\label{eq:basicSparse}
\begin{split}
  \int\Big(\sum_{S\in\mathscr S} 1_S c_S\Big)\phi
  &\lesssim \sum_{S\in\mathscr S} c_S \int_{E_S} \ave{\phi}_S
  \leq\int\Big(\sum_{S\in\mathscr S} 1_{E(S)}c_S\Big)M\phi \\
  &\leq\Big\|\sum_{S\in\mathscr S} 1_{E(S)}c_S\Big\|_{L^{r}}\Norm{M\phi}{L^{r'}}
  \lesssim\Big(\sum_{S\in\mathscr S}\abs{E(S)} c_S^r\Big)^{1/r}\Norm{\phi}{L^{r'}}.
\end{split}
\end{equation}
Dualizing the $\ell^r$ norm with $\ell^{r'}$, we find that
\begin{equation*}
  \Big(\sum_{S\in\mathscr S}\Big[\abs{S}^{1/r} \fint_S\abs{b-\ave{b}_S}\Big]^r\Big)^{1/r}
  =\sum_{S\in\mathscr S}\wt \lambda_S \abs{S}^{1/r}\fint_S\abs{b-\ave{b}_S}
  =\sum_{S\in\mathscr S}\lambda_S\int_S\abs{b-\ave{b}_S}
\end{equation*}
for suitable coefficients $\lambda_S$ with $\sum_{S\in\mathscr S}\abs{S}\lambda_S^{r'} = \sum_{S\in\mathscr S}\wt \lambda_S^{r'} \le 1$.

Next, using sparseness and the definition of the sharp maximal operator, we observe that
given an arbitrary sparse collection $\mathscr S$ and coefficients $\lambda_S$ with $\sum_{S\in\mathscr S}\abs{S}\lambda_S^{r'} \le 1$, we have
\begin{equation*}
\begin{split}
  \sum_{S\in\mathscr S}\abs{S}\lambda_S\fint_S\abs{b-\ave{b}_S}
  &\lesssim\sum_{S\in\mathscr S}\abs{E(S)}\lambda_S\inf_{z\in S}M^{\#}b(z)
  \leq\int\Big(\sum_{S\in\mathscr S} 1_{E(S)}\lambda_S\Big)M^{\#}b \\
  &\leq\BNorm{\sum_{S\in\mathscr S}1_{E(S)}\lambda_S}{L^{r'}}\Norm{M^{\#}b}{L^r}
  =\Big(\sum_{S\in\mathscr S}\abs{E(S)}\lambda_S^{r'}\Big)^{1/r'}\Norm{M^{\#}b}{L^r},
\end{split}
\end{equation*}
where the first factor is bounded by $1$. Lastly, we have
\begin{equation*}
  \Norm{M^{\#}b}{L^r}=\Norm{M^{\#}(b-c)}{L^r}\lesssim\Norm{b-c}{L^r}
\end{equation*}
for every constant $c$, and hence $\Norm{M^{\#}b}{L^r}\lesssim \Norm{b}{\dot L^r}$.
\end{proof}

\subsection{Bi-parameter spaces}
For $b \in L^1_{\loc}$ and a rectangle $R = I_1 \times I_2$ we denote
\begin{equation}\label{eq:osc}
  \osc^{v_1,v_2}(b,R)
  :=\Norm{b-\langle b\rangle_{I_1,1} -\langle b\rangle_{I_2,2} +\langle b\rangle_R}{L^{v_1}_{x_1}L^{v_2}_{x_2}(R)}, \qquad 1 \le v_i \le \infty.
\end{equation}

\subsubsection*{Homogeneous H\"older spaces}\label{subsec:bihol}
If $X$ is a Banach space with norm $|\cdot|_X$ and $b \colon \R^d \to X$, we define
$$
\|b\|_{\dot C^{0, \alpha}(X)} = \sup_{x \ne y} \frac{|b(x) - b(y)|_X}{|x-y|^{\alpha}}.
$$
In $\R^d = \R^{d_1} \times \R^{d_2}$ the natural bi-parameter homogeneous H\"older norm is
\begin{align*}
\|b\|_{\dot C^{0, \alpha, \beta}(\R^d)} &:= \|b\|_{\dot C^{0, \alpha}_{x_1}(\dot C^{0, \beta}_{x_2}) } = \|b\|_{\dot C^{0, \beta}_{x_2}(\dot C^{0, \alpha}_{x_1}) }  \\
&= \sup_{x_1 \ne y_1} \Big\|\frac{b(x_1, \cdot)- b(y_1, \cdot)}{|x_1-y_1|^{\alpha}} \Big\|_{\dot C^{0, \beta}}
= \sup_{x_2 \ne y_2} \Big\|\frac{b(\cdot, x_2)- b(\cdot, y_2)}{|x_2-y_2|^{\beta}} \Big\|_{\dot C^{0, \alpha}} \\
&=
 \sup_{\substack{x_1\neq y_1\\ x_2\neq y_2}}
\frac{|b(x_1, x_2)-b(x_1, y_2)-b(y_1,x_2)+b(y_1, y_2)|}{|x_1-y_1|^{\alpha}|x_2-y_2|^{\beta}}.
\end{align*}
The following oscillatory characterization holds for this norm.
\begin{prop}\label{prop:calphacbeta}
We have
\begin{equation*}
\begin{split}
\|b\|_{\dot C^{0, \alpha, \beta}(\R^d)}
  \sim\sup\Big\{ \frac{1}{\ell(I_1)^\alpha\ell(I_2)^\beta} \frac{\osc^{1,1}(b, R)}{|R|} \colon R = I_1 \times I_2 \textup{ rectangle}
\Big\}.
\end{split}
\end{equation*}
\end{prop}
\begin{proof}
We denote the supremum on the right hand side by $N$.

``$\gtrsim$'' is immediate, since for all $x_i \in I_i$ we have
\begin{align*}
\abs{b(x_1,x_2)-&\ave{b}_{I_1,1}(x_2)-\ave{b}_{I_2,2}(x_1)+\ave{b}_{I_1\times I_2}} \\
&\le \fint_{I_1\times I_2}|b(x_1, x_2)-b(x_1,y_2)-b(y_1,x_2)+b(y_1,y_2)|\ud y_1\ud y_2 \\
&\lesssim \|b\|_{\dot C^{0, \alpha, \beta}(\R^d)} \ell(I_1)^\alpha \ell(I_2)^\beta.
\end{align*}

For ``$\lesssim$'', notice that by Proposition \ref{prop:CalphaOsc} we have
\begin{align*}
\|b\|_{\dot C^{0, \alpha, \beta}(\R^d)} &= \sup_{x_2 \ne y_2} \Big\|\frac{b(\cdot, x_2)- b(\cdot, y_2)}{|x_2-y_2|^{\beta}} \Big\|_{\dot C^{0, \alpha}} \\
&\sim \sup_{x_2 \ne y_2} \frac{1}{|x_2-y_2|^{\beta}}
 \sup_{I_1} \frac{1}{\ell(I_1)^\alpha} \fint_{I_1} |b_{x_2, y_2}-\langle b_{x_2, y_2} \rangle_{I_1}|,
\end{align*}
where
\[
b_{x_2, y_2}=b(\cdot, x_2)- b(\cdot, y_2).
\]
So we fix $x_2 \ne y_2$ and a cube $I_1 \subset \R^{d_1}$.
By expanding $b(\cdot, x_2)$ as in the proof of Proposition \ref{prop:CalphaOsc} (using the same notation as there) we get
\begin{equation*}
  b(\cdot, x_2)=\sum_{k=1}^\infty\Big(\fint_{Q_{k+1}(x_2)}b(\cdot, z_2)\ud z_2-\fint_{Q_k(x_2)}b(\cdot, z_2)\ud z_2\Big)+\fint_{Q(\frac12(x_2+y_2),\frac12\abs{x_2-y_2})}b(\cdot, z_2)\ud z_2.
\end{equation*}
Hence
\begin{align*}
b_{x_2, y_2}&= \sum_{k=1}^\infty \big[ \langle b\rangle_{Q_{k+1}(x_2), 2} - \langle b\rangle_{Q_k(x_2), 2}\big]
 - \sum_{k=1}^\infty \big[ \langle b\rangle_{Q_{k+1}(y_2), 2} - \langle b\rangle_{Q_k(y_2), 2}\big].
\end{align*}
We first consider the contribution of the first summand. We have
\begin{align*}
& \fint_{I_1}\left|\big[ \langle b\rangle_{Q_{k+1}(x_2), 2} - \langle b\rangle_{Q_k(x_2), 2}\big]-\bla\big[ \langle b\rangle_{Q_{k+1}(x_2), 2} - \langle b\rangle_{Q_k(x_2), 2}\big]\bra_{I_1}\right|\\
&= \fint_{I_1}\left|\langle b-\langle b\rangle_{I_1,1}\rangle_{Q_{k+1}(x_2), 2} - \langle b-\langle b\rangle_{I_1,1}\rangle_{Q_k(x_2), 2}\right|\\
&\lesssim  \fint_{I_1 \times{Q_k(x_2)}}\big|  b-\langle b\rangle_{I_1,1} - \langle b-\langle b\rangle_{I_1,1}\rangle_{Q_k(x_2), 2}\big|\\
&\lesssim N\ell(I_1)^{\alpha} (2^{-k}|x_2-y_2|)^{\beta}.
\end{align*}
Thus, the contribution of the first summand to the average $\fint_{I_1} |b_{x_2, y_2}-\langle b_{x_2, y_2} \rangle_{I_1}|$,
can be dominated by $N\ell(I_1)^{\alpha} |x_2-y_2|^{\beta}$. Likewise, the second summand contributes the same. Thus, we have established
$$
\fint_{I_1} |b_{x_2, y_2}-\langle b_{x_2, y_2} \rangle_{I_1}| \lesssim   N \ell(I_1)^{\alpha} |x_2-y_2|^{\beta}
$$
proving that $\|b\|_{\dot C^{0, \alpha, \beta}(\R^d)} \lesssim N$ as desired.
\end{proof}
For the norm
$$
\|b\|_{\dot C^{0, \alpha}_{x_1}(\BMO_{x_2})}=\sup_{x_1\neq y_1} \left\| \frac{b(x_1, \cdot)-b(y_1, \cdot)}{|x_1-y_1|^{\beta_1}}\right\|_{\BMO(\R^{d_2})}
$$
the following oscillatory characterization holds.
\begin{prop}\label{prop:CBMOchar}
We have
\begin{equation*}
\begin{split}
  \Norm{b}{\dot C^{0,\alpha}_{x_1}(\BMO_{x_2})}
  \sim\sup\Big\{ \frac{1}{\ell(I_1)^\alpha}\frac{\osc^{1,1}(b, R)}{|R|} \colon R = I_1 \times I_2 \textup{ rectangle}\Big\}.
\end{split}
\end{equation*}
\end{prop}

\begin{proof}
We denote the supremum on the right hand side by $N$.

We again first prove the ``$\gtrsim$'' direction. For all $x_i \in I_i$ we have
\begin{align*}
\abs{b(x_1, x_2)-&\ave{b}_{I_1,1}(x_2)-\ave{b}_{I_2,2}(x_1)+\ave{b}_{I_1\times I_2}} \\ &= \left|\fint_{I_1\times I_2}\big(b(x_1, x_2)-b(x_1,y_2)-b(y_1,x_2)+b(y_1,y_2)\big)\ud y_1\ud y_2\right|\\
&=\left|\fint_{I_1}|x_1-y_1|^\alpha \Big(\frac{b(x_1, x_2)-b(y_1,x_2)}{|x_1-y_1|^\alpha}-\Big \langle\frac{b(x_1,\cdot)-b(y_1,\cdot)}{|x_1-y_1|^\alpha}\Big \rangle_{I_2}\Big)\ud y_1\right|\\
&\lesssim \ell(I_1)^\alpha\fint_{I_1}  \Big|\frac{b(x_1, x_2)-b(y_1,x_2)}{|x_1-y_1|^\alpha}-\Big \langle \frac{b(x_1,\cdot)-b(y_1,\cdot)}{|x_1-y_1|^\alpha}\Big \rangle_{I_2}\Big|\ud y_1.
\end{align*}Then taking the average over $I_1\times I_2$ and using Fubini we obtain the estimate as desired.

For ``$\lesssim$'' we fix $x_1, y_1\in \R^{d_1}$ with $x_1\neq y_1$. Similarly as Proposition \ref{prop:calphacbeta} (just let $\alpha=0$ there) we see that for all cubes $I_2 \subset \R^{d_2}$
\[
\fint_{I_2} |b_{x_1, y_1}-\langle b_{x_1, y_1} \rangle_{I_2}| \lesssim   N |x_1-y_1|^{\alpha},
\]where $b_{x_1, y_1}=b(x_1,\cdot)-b(y_1,\cdot)$. This proves the claim.
\end{proof}

The final $\dot C^{0, \alpha}(X)$ type space of interest to us is the bi-parameter space $\dot C^{0, \alpha}_{x_1}(\dot L^r_{x_2})$.

\begin{prop}\label{prop:CLr}
We have
\begin{equation*}
\begin{split}
  \Norm{b}{\dot C^{0,\alpha}_{x_1}(\dot L^r_{x_2})}
  \sim\sup\Big\{ \frac{1}{\abs{I_1}\ell(I_1)^\alpha}\sum_{S_2 \in\mathscr S_2}\lambda_{S_2}& \osc^{1,1}(b, I_1 \times S_2)\colon \\
  &    I_1 \subset\R^{d_1},\mathscr S_2 \text{ sparse},\sum_{S_2 \in\mathscr S_2 }\abs{S_2}\lambda_{S_2}^{r'}\leq 1\Big\},
\end{split}
\end{equation*}
where $I_1 \subset \R^{d_1}$  is a cube and $\mathscr{S}_2$ is a sparse collection of cubes in $\R^{d_2}$.
\end{prop}

\begin{proof}
``$\gtrsim$'': We apply Proposition \ref{prop:LrOsc} to the function $b(x_1,\cdot)-\ave{b}_{I_1,1}$ for each $x_1 \in I_1$. This shows that
\begin{equation*}
\begin{split}
  \sum_{S_2\in\mathscr S_2}\lambda_{S_2} &\int_{S_2}\abs{b(x_1,\cdot)-\ave{b}_{I_1,1}-\ave{b}_{S_2,2}(x_1)+\ave{b}_{I_1\times S_2}} \\
  &\lesssim\Norm{b(x_1,\cdot)-\ave{b}_{I_1,1}}{\dot L^r_{x_2}}
  \lesssim \fint_{I_1} \Norm{b(x_1,\cdot)-b(y_1,\cdot)}{\dot L^r_{x_2}}\ud y_1
  \lesssim \Norm{b}{\dot C^{0,\alpha}_{x_1}(\dot L^r_{x_2})}\ell(I_1)^\alpha.
\end{split}
\end{equation*}
The penultimate estimate is easy to see using Proposition \ref{prop:LrOsc}.
Integrating over $x_1 \in I_1$, we deduce the claimed bound.

``$\lesssim$'': Denoting by $N$ the right-hand side, we need to prove that $\Norm{b(x_1,\cdot)-b(y_1,\cdot)}{\dot L_{x_2}^r}\lesssim N\abs{x_1-y_1}^\alpha$ for all $x_1\neq y_1 \in\R^{d_1}$. By Proposition \ref{prop:LrOsc}, this is equivalent to
\begin{equation}\label{eq:LrOscToProve}
   \sum_{S_2\in\mathscr S_2}\lambda_{S_2} \int_{S_2}\abs{b(x_1,\cdot)-b(y_1,\cdot)-\ave{b(x_1,\cdot)-b(y_1,\cdot)}_{S_2}}\lesssim N\abs{x_1-y_1}^\alpha,
\end{equation}
where $\mathscr S_2$ and $\lambda_{S_2}$ are as in the claim. We then expand $b(x_1,\cdot)$ as in the proof of Proposition \ref{prop:CalphaOsc} (using the same notation as there):
\begin{equation*}
  b(x_1,\cdot)=\sum_{k=1}^\infty\Big(\fint_{Q_{k+1}(x_1)}b(z_1,\cdot)\ud z_1-\fint_{Q_k(x_1)}b(z_1,\cdot)\ud z_1\Big)+\fint_{Q(\frac12(x_1+y_1),\frac12\abs{x_1-y_1})}b(z_1,\cdot)\ud z_1.
\end{equation*}
Doing the same with $b(y_1,\cdot)$ and using the triangle inequality, we find that
\begin{equation*}
\begin{split}
  &LHS\eqref{eq:LrOscToProve} \\
  &\leq\sum_{u_1=x_1,y_1}\sum_{k=1}^\infty\sum_{S_2\in\mathscr S_2}\lambda_{S_2} \int_{S_2}\abs{\ave{b}_{Q_{k+1}(u_1),1}-\ave{b}_{Q_k(u_1),1}-(\ave{b}_{Q_{k+1}(u_1)\times S_2}-\ave{b}_{Q_k(u_1)\times S_2})} \\
  &\leq\sum_{u_1=x_1,y_1}\sum_{k=1}^\infty\sum_{S_2\in\mathscr S_2}\lambda_{S_2} \int_{S_2}\fint_{Q_{k+1}(u_1)}\abs{b-\ave{b}_{Q_k(u_1),1}-(\ave{b}_{S_2,2}-\ave{b}_{Q_k(u_1)\times S_2})} \\
  &\lesssim\sum_{u_1=x_1,y_1}\sum_{k=1}^\infty\sum_{S_2\in\mathscr S_2}\lambda_{S_2} \int_{S_2}\fint_{Q_{k}(u_1)}\abs{b-\ave{b}_{Q_k(u_1),1}-\ave{b}_{S_2,2}+\ave{b}_{Q_k(u_1)\times S_2}} \\
  &\leq\sum_{u_1=x_1,y_1}\sum_{k=1}^\infty N\ell(Q_k(u_1))^\alpha \lesssim \sum_{u_1=x_1,y_1}\sum_{k=1}^\infty N (2^{-k}\abs{x_1-y_1})^\alpha \lesssim N\abs{x_1-y_1}^\alpha,
\end{split}
\end{equation*}
and this proves \eqref{eq:LrOscToProve}.
\end{proof}
\begin{rem}
Note that it seems somewhat important that in the above proof we apply Proposition \ref{prop:LrOsc} directly to the defining condition $\Norm{b(x_1,\cdot)-b(y_1,\cdot)}{\dot L_{x_2}^r}\lesssim\abs{x_1-y_1}^\alpha$, so that the extremizing $\mathscr S_2$ and $\lambda_{S_2}$ depend on $x_1,y_1$ only, and only then apply the (considerations of) Proposition \ref{prop:CalphaOsc}. Alternatively, it might occur to one to start with (a vector-valued version of) Proposition \ref{prop:CalphaOsc}, reducing the proof to estimating $\fint_{I_1}\Norm{b(z_1,\cdot)-\ave{b}_{I_1,1}}{\dot L_{x_2}^r}\ud z_1\lesssim\ell(I_1)^\alpha$. If we now tried to continue with Proposition \ref{prop:LrOsc}, applied to $b(z_1,\cdot)-\ave{b}_{I_1,1}$ for each $z_1\in I_i$, the resulting $\mathscr S_2$ and $\lambda_{S_2}$ would in general depend on the integration parameter $z_1\in I_1$.
\end{rem}

\subsubsection*{Bounded mean oscillation}
We define, for $1 \le p_i < \infty$, the rectangular $\BMO$ norm
$$
\|b\|_{\BMO_{\textup{rect}}(p_1, p_2)} = \sup_{R=I_1 \times I_2} \frac{\osc^{p_1,p_2}(b,R)}{|I_1|^{1/{p_1}} |I_2|^{1/{p_2}}}.
$$
This space is directly in the oscillatory form. It does not appear to enjoy the John--Nirenberg property, and so the choice of the exponents can matter. Usually
in the literature $$\|b\|_{\BMO_{\textup{rect}}} := \|b\|_{\BMO_{\textup{rect}}(2, 2)}.$$ For many purposes this is not the correct or optimal bi-parameter $\BMO$ space.

We will later have use for the smaller product BMO space, but it does not enjoy an oscillatory characterisation, and we will not discuss its definition now.

If $X$ is a Banach space, we say that a locally integrable $b \colon \R^d \to X$
belongs to the $X$-valued BMO space $\BMO(p,X)$, $0 < p < \infty$, if
$$
\|b\|_{\BMO(p, X)} := \sup_{I} \Big( \fint_{I} |b - \langle b \rangle_{I}|_X^p \Big)^{1/p}
  < \infty.
$$
By the Banach-valued John--Nirenberg theorem (see e.g. \cite[Theorem 3.2.30]{HNVW1}) we have that all these norms are equivalent when $p$ varies, and we may again
set
$$
\|b\|_{\BMO(X)} := \sup_{I}  \fint_{I} |b - \langle b \rangle_{I}|_X.
$$

From this point on, we do not obtain oscillatory characterization for our spaces, but we can record one-sided estimates
of oscillatory quantities by function space norms. These are of philosophical use.
Later, we may obtain
commutator lower bounds with respect to these potentially smaller oscillatory quantities.
We start with the bi-parameter space $\BMO_{x_1}(\dot L^{r}_{x_2})$.
\begin{lem}\label{lem:BMOLr}
We have
\begin{equation*}
\begin{split}
  \sup\Big\{ \frac{1}{\abs{I_1}} \sum_{S_2 \in\mathscr S_2}\lambda_{S_2} \osc^{1,1}(b, I_1 \times S_2) \colon 
     I_1 \subset\R^{d_1},\mathscr S_2 \text{ sparse},&\sum_{S_2 \in\mathscr S_2 }\abs{S_2}\lambda_{S_2}^{r'}\leq 1\Big\} \\
     &\lesssim \|b\|_{\BMO_{x_1}(\dot L^{r}_{x_2})},
\end{split}
\end{equation*}
where $I_1 \subset \R^{d_1}$  is a cube and $\mathscr{S}_2$ is a sparse collection of cubes in $\R^{d_2}$.
\end{lem}
\begin{proof}
As in Proposition \ref{prop:CLr} we have by applying Proposition \ref{prop:LrOsc} to $b(x_1,\cdot)-\ave{b}_{I_1,1}$ that
\begin{align*}
\fint_{I_1} \sum_{S_2\in\mathscr S_2}\lambda_{S_2} \int_{S_2}\abs{b(x_1,\cdot)-&\ave{b}_{I_1,1}-\ave{b}_{S_2,2}(x_1)+\ave{b}_{I_1\times S_2}} \\
  &\lesssim \fint_{I_1} \Norm{b(x_1,\cdot)-\ave{b}_{I_1,1}}{\dot L^r_{x_2}} \lesssim \|b\|_{\BMO_{x_1}(\dot L^{r}_{x_2})}.\qedhere
\end{align*}
\end{proof}

\subsubsection*{The space $\dot L^{r_1} \dot L^{r_2}$}
The final condition that comes up in our commutator lower bounds is
\begin{align*}
N := \sup \sum_{S_1 \in\mathscr S_1}\sum_{S_2 \in\mathscr S_2} \lambda_{1,S_1}\lambda_{2,S_2}\osc^{1,1}(b,S_1\times S_2),
\end{align*}
where $r_i \in (1,\infty)$ and the supremum is taken over all sparse collections $\mathscr S_i$ of cubes in $\R^{d_i}$ and non-negative coefficients $\lambda_{i, S_i}$ satisfying $\sum_{S_i\in\mathscr S_i}\lambda_{i, S_i}^{r_i'}|S_i| \le 1$. We again only have the direction that
\begin{equation}\label{eq:Lrr}
  N \lesssim \|b\|_{\dot L^{r_1}_{x_1} \dot L^{r_2}_{x_2}},
\end{equation}
which is primarily of philosophical use. This can be seen by utilising Proposition \ref{prop:LrOsc} and its proof.

\section{Weak factorization and corollaries}\label{sec:WF}

\begin{defn}[Non-degenerate kernels]
Let $K_i$ be a Calder\'on-Zygmund kernel on $\R^{d_i}$.
We say that $K_i$ is {\em (symmetrically) non-degenerate} if there is a constant $c_0 > 0$ such that for every $y_i\in\R^{d_i}$ and $r>0$ there exists $x_i\in B(y_i,r)^c$ so that there holds
\[
|K_i(x_i,y_i)|\ge \frac 1{c_0 r^{d_i}}\quad\Big(\text{and}\quad |K_i(y_i,x_i)|\ge \frac 1{c_0 r^{d_i}}\Big).
\]
\end{defn}

Notice that the $x_i$ in the above definition necessarily satisfies $|x_i-y_i| \sim r$.
If $T_i$ is a CZO whose kernel $K_i$ is a non-degenerate Calder\'on-Zygmund kernel, then we say that $T_i$ is a non-degenerate CZO. However,
in what follows we do not really need CZOs, we only need the kernels, and $T_1$ and $T_1^*$ will just be notation for the integrals
\begin{align*}
T_1 h(x_1) = \int_{\R^{d_1}} K_1(x_1, y_1)  h(y_1)\ud y_1, \qquad x_1 \in \R^{d_1} \setminus \operatorname{spt} h, \\
T_1^* h(x_1) = \int_{\R^{d_1}} K_1(y_1, x_1)  h(y_1)\ud y_1, \qquad x_1 \in \R^{d_1} \setminus \operatorname{spt} h,
\end{align*}
and similarly for $T_2$.

\begin{defn}[The ``reflected'' cube $\wt I_i$]\label{def:tildeI}
Let $K_i$ be a fixed non-degenerate Calder\'on-Zygmund kernel on $\R^{d_i}$, and fix a large constant $A \ge 3$. For each cube $I_i \subset\R^{d_i}$ with center $c_{I_i}$ and sidelength $\ell(I_i)$, we define another cube $\wt I_i\subset\R^{d_i}$ of the same size by choosing a center $c_{\wt I_i}$, guaranteed by the non-degenaracy of $K_i$, so that
\begin{equation*}
  |c_{I_i}-c_{\wt I_i}| \sim A\ell(I_i),\qquad |K_i(c_{\wt I_i}, c_{I_i})| \sim (A \ell(I_i))^{-d_i}.
\end{equation*}
Notice that $\dist(I_i,\wt I_i)\sim A\ell(I_i)$. If $K_i$ is symmetrically non-degenerate, we require in addition that $|K_i(c_{I_i}, c_{\wt I_i})| \sim (A \ell(I_i))^{-d_i}$. While the choice of $\wt I_i$ may not be unique in general, in the symmetric case we can and will make it in such a way that $\,\wt{\!\wt I_i}=I_i$.
\end{defn}

By (the proof) of Proposition 2.2 in \cite{Hy5} we have for all $x_i \in \wt I_i$ and $y_i \in I_i$ that
\begin{equation}\label{eq:Hol1}
|K_i(x_i, y_i) - K_i(c_{\wt I_i}, c_{I_i})| \le \frac{CA^{-\alpha_i}}{(A \ell(I_i))^{d_i}}.
\end{equation}

If $K_1$ and $K_2$ are two (symmetrically) non-degenerate Calder\'on--Zygmund kernels on $\R^{d_1}$ and $\R^{d_2}$, and $I_i \subset \R^{d_i}$ are cubes, we of course define $\wt I_i$ with respect to $K_i$. For each rectangle $R=I_1\times I_2\subset\R^d = \R^{d_1} \times \R^{d_2}$, we define $\wt R:=\wt I_1\times \wt I_2$.

For a rectangle $R=I_1\times I_2\subset  \R^{d_1+d_2} = \R^d$ and $p_1, p_2 \in [1,\infty]$ we denote
\begin{equation*}
L^{p_1}_{x_1, 0}L^{p_2}_{x_2, 0}(R)=\Big\{f\in L^{p_1}_{x_1}L^{p_2}_{x_2}(R): \int_{I_1} f = \int_{I_2} f = 0\Big\}.
\end{equation*}

We record and prove the following bi-parameter weak factorization theorem. A simpler one-parameter version is
Lemma 2.5 of \cite{Hy5}. This will have important corollaries that we use to deduce the commutator lower bounds.

\begin{lem}\label{lem:IteratedFactorization} Let $K_i$ be a non-degenerate kernel on $\R^{d_i}$, $i=1,2$.
If $R = I_1 \times I_2 \subset \R^d$ is a rectangle and $f\in L^1_{00}(R)$, then
we can decompose
\begin{equation*}
\begin{split}
	f &= 1_{\wt R} T_1T_2 h-T_2 h \cdot T_1^*1_{\wt R}-T_1h\cdot T_2^*1_{\wt R}+h T_1^*T_2^*1_{\wt R} + \sum_{j=1}^3\tilde{f}_j, \\
\end{split}
\end{equation*}
where each $\tilde{ f}_j \in L^1_{00}(R_j)$ for $R_1 = \wt I_1\times I_2$, $R_2 =  I_1 \times \wt I_2$ and $R_3 = \wt R$, and they satisfy
$$
\abs{\tilde f_1(x)} \lesssim A^{-\alpha_1} 1_{\wt I_1}(x_1) \langle |f| \rangle_{I_1,1}(x_2), \quad \abs{\tilde f_2(x)} \lesssim A^{-\alpha_2} \langle |f| \rangle_{I_2,2}(x_1)1_{\wt I_2}(x_2)
$$
and
$$
\abs{\tilde f_3(x)} \lesssim A^{-\alpha_1} A^{-\alpha_2} 1_{\wt R}(x) \langle |f| \rangle_{R}.
$$
Moreover, $h \in L^1(R)$ satisfies
$$
\abs{h(x)} \lesssim A^d \abs{f(x)}.
$$
In particular, if $f \in L^{u_1}_{x_1, 0}L^{u_2}_{x_2, 0}(R)$, $1 \le u_i \le \infty$, then also
$$
\|h\|_{L^{u_1}_{x_1}L^{u_2}_{x_2}} \lesssim A^d \|f\|_{L^{u_1}_{x_1}L^{u_2}_{x_2}} \qquad \textup{and} \qquad \Norm{\tilde{f}_j}{L^{u_1}_{x_1}L^{u_2}_{x_2}} \lesssim A^{-\min(\alpha_1, \alpha_2)} \Norm{f}{L^{u_1}_{x_1}L^{u_2}_{x_2}},
$$
where the implicit constants depend only on the dimensions and the kernel constants.
\end{lem}
\begin{proof}
Denote $\phi=1_{\wt R}=1_{\wt I_1}\otimes 1_{\wt I_2}=: \phi_1\otimes \phi_2$. We write
\begin{equation*}
\begin{split}
f = \frac{f}{T_1^*T_2^*\phi}T_1^*T_2^*\phi &= \phi T_1T_2h -T_2h\cdot T_1^*\phi-T_1h\cdot T_2^*\phi+h T_1^*T_2^*\phi + \tilde{f},
\end{split}
\end{equation*}
where $$h:=\frac{f}{T_1^*T_2^*\phi}$$ and
\begin{equation*}
\begin{split}
\tilde{f} &= T_1h\cdot T_2^*\phi+T_2h\cdot T_1^*\phi- \phi T_1T_2h \\
&= \phi_1T_1\Big(\frac{f}{T_1^*\phi_1}\Big)+\phi_2T_2\Big(\frac{f}{T_2^*\phi_2}\Big) - \phi T_1\Big( \frac{1}{T_1^*\phi_1}T_2\Big(\frac{f}{T_2^*\phi_2}\Big) \Big) \\
&= \tilde{f}_1+\tilde{f}_2+\tilde{f}_3.
\end{split}
\end{equation*}
Everything is well-defined, since in fact we have the estimates $T_i^*1_{\wt I_i}(x_i) \gtrsim A^{-d_i}$ for $x \in I_i$ (remember that $f$ is supported on $R$). To see this, fix $x_1\in I_1$ and write
\begin{align*}
T_1^*1_{\wt I_1}(x_1)&=\int_{\wt I_1} K_1(y_1,x_1) \ud y_1\\
&= K_1(c_{\wt I_1}, c_{I_1})|I_1|+\int_{\wt I_1} \big(K_1(y_1,x_1)-K_1(c_{\wt I_1}, c_{I_1})\big) \ud y_1.
\end{align*}
Using \eqref{eq:Hol1} it is then easy to see that $|T_1^*1_{\wt I_1}(x_1)|\gtrsim A^{-d_1}$. The other estimate is symmetric.
This also gives that
$$
\abs{h(x)} \lesssim A^d \abs{f(x)}.
$$

It is also immediate that $\int_{I_2}\tilde f_1=0$ and
\begin{equation*}\label{wfp1}
\int_{{\wt I_1}}\tilde{f}_1 =\int_{\R^{d_1}} \phi_1T_1\Big(\frac{f}{T_1^*\phi_1}\Big)=\int_{\R^{d_1}}T_1^*\phi_1\frac{f}{T_1^*\phi_1}=\int_{I_1}f=0 .
\end{equation*} Similarly, we can see that $\int_{I_1}\tilde f_2=\int_{\wt I_2} \tilde f_2=\int_{\wt I_1} \tilde f_3=\int_{\wt I_2} \tilde f_3=0$.

We now prove the poinwise estimates of $\tilde f_i$.
Let $g \colon \R^d \to \C$ satisfy $g(x) = 0$ if $x_1 \not \in I_1$ and $\int_{I_1} g = 0$.
For $x = (x_1, x_2) \in \wt I_1 \times \R^{d_2}$ we write
$$
T_1\Big( \frac{g}{T_1^*\phi_1}\Big)(x) = T_1\Big( \frac{g(\cdot, x_2)}{T_1^*\phi_1} - \frac{g(\cdot, x_2)}{|I_1| K_1(c_{\wt I_1}, c_{I_1})} \Big)(x_1) + \frac{1}{|I_1| K_1(c_{\wt I_1}, c_{I_1})} T_1(g(\cdot, x_2))(x_1).
$$
We have using \eqref{eq:Hol1} that
\begin{align*}
\Big| \frac{1}{|I_1| K_1(c_{\wt I_1}, c_{I_1})} T_1(g(\cdot, x_2))(x_1) \Big| &\sim A^{d_1} \Big|\int_{I_1} [K_1(x_1, y_1) - K_1(c_{\wt I_1}, c_{I_1})] g(y_1,x_2)\ud y_1 \Big| \\
&\lesssim A^{-\alpha_1} \fint_{I_1} |g(\cdot, x_2)|.
\end{align*}
It is similarly straightforward that
$$
\Big| T_1\Big( \frac{g(\cdot, x_2)}{T_1^*\phi_1} - \frac{g(\cdot, x_2)}{|I_1| K_1(c_{\wt I_1}, c_{I_1})} \Big)(x_1) \Big| \lesssim A^{-\alpha_1} \fint_{I_1} |g(\cdot, x_2)|.
$$
It follows that
\begin{equation}\label{eq:eq1}
\Big|T_1\Big( \frac{g}{T_1^*\phi_1}\Big)(x)\Big| \lesssim A^{-\alpha_1} \fint_{I_1} |g(\cdot, x_2)|, \qquad x \in \wt I_1 \times \R^{d_2}.
\end{equation}
Applying this to $g = f$ we get
$$
\abs{\tilde f_1(x)} \lesssim A^{-\alpha_1} 1_{\wt I_1}(x_1) \langle |f| \rangle_{I_1,1}(x_2).
$$
The estimate
$$
\abs{\tilde f_2(x)} \lesssim A^{-\alpha_2} \langle |f| \rangle_{I_2,2}(x_1)1_{\wt I_2}(x_2)
$$
is symmetric. Lastly, notice that
$$
\tilde f_3 = \phi_1T_1\Big( \frac{1}{T_1^*\phi_1}\phi_2T_2\Big(\frac{f}{T_2^*\phi_2}\Big) \Big)
= \phi_1T_1\Big( \frac{\tilde f_2}{T_1^*\phi_1}\Big)
$$
so that applying \eqref{eq:eq1} to $g = \tilde f_2$ (which is a legitimate choice) and the estimate of $\tilde f_2$, we get that for $x \in \wt I_1 \times \wt I_2$ there holds
$$
|\tilde f_3(x)| \lesssim  A^{-\alpha_1} \fint_{I_1} |\tilde f_2(y_1, x_2)| \ud y_1 \lesssim A^{-\alpha_1} A^{-\alpha_2} \fint_R |f|.
$$
The claimed $L^{u_1}L^{u_2}$ estimates are clear from the pointwise estimates.
\end{proof}
\begin{rem}\label{rem:modcon}
A much weaker modulus of continuity like in \cite{Hy5} for the kernels $K_1, K_2$ would also work, but we prefer to use the standard estimates here for simplicity.
The weaker assumption for $K_1$ would be
$$
|K_1(x_1, y_1) - K_1(x_1', y_1)| + |K_1(y_1, x_1) - K_1(y_1, x_1')| \le \frac{1}{|x_1-y_1|^{d_1}} \omega\Big( \frac{|x_1-x_1'|}{|x_1-y_1|} \Big)
$$
whenever $|x_1-x_1'| \le |x_1-y_1|/2$. In these lower bound arguments it is only needed that the increasing modulus of continuity
$\omega\colon [0,1) \to [0, \infty)$ satisfies $\omega(t) \to 0$ when $t \to 0$.
\end{rem}

The previous lemma will be used to control the oscillations \eqref{eq:osc} via the following basic observation. The idea
is to decompose the dualizing $f$ like above.
\begin{lem}\label{lem:oscDual}
For $v_1,v_2\in[1,\infty]$, we have
\begin{equation*}
  \osc^{v_1,v_2}(b,R)
  \sim \sup\Big\{\Babs{\int_{R}bf}: \Norm{f}{L^{v_1'}_{x_1,0} L^{v_2'}_{x_2,0} (R)}\leq 1\Big\}.
\end{equation*}
\end{lem}

\begin{proof}
By standard dualities, we have
\begin{equation*}
\begin{split}
  \osc^{v_1,v_2}(b,R)
  &=\sup\Big\{\Babs{\iint_{I_1\times I_2}[b-\ave{b}_{I_1,1}-\ave{b}_{I_2,2}+\ave{b}_{R}]g}:\Norm{g}{L^{v_1'}L^{v_2'}(R)}\leq 1\Big\} \\
  &=\sup\Big\{\Babs{\iint_{I_1\times I_2}b[g-\ave{g}_{I_1,1}-\ave{g}_{I_2,2}+\ave{g}_{R}]}:\Norm{g}{L^{v_1'}L^{v_2'}(R)}\leq 1\Big\} \\
  &\leq \sup\Big\{\Babs{\iint_{I_1\times I_2}bg}: \Norm{g}{L^{v_1'}_0 L^{v_2'}_0 (R)}\leq 4\Big\} \\
  &=\sup\Big\{\Babs{\iint_{I_1\times I_2}[b-\ave{b}_{I_1,1}-\ave{b}_{I_2,2}+\ave{b}_{R}]g}: \Norm{g}{L^{v_1'}_0 L^{v_2'}_0 (R)}\leq 4\Big\} \\
  &\leq 4\osc^{v_1,v_2}(b,R),
\end{split}
\end{equation*}
so that all these quantities are actually comparable.
\end{proof}

Motivated by this we define the following truncated oscillation, which is finite even when $b \in L^1_{\loc}$. This is important for technical reasons -- the absorption argument
below requires a finite quantity.
For $b \in L^1_{\loc}$ and $N \in \N$ we define
$$
\osc^{v_1,v_2}_N(b,R)
  := \sup\Big\{\Babs{\iint_{R}bf}: \Norm{f}{L^{v_1'}_{0} L^{v_2'}_{0} (R)}\leq 1 \textup{ and }  \|f\|_{L^{\infty}} \le N \Big\} < \infty.
$$
\begin{rem}\label{rem:osc11}
When $v_1=v_2=1$ -- which is the case of most interest -- we can directly work with $\osc^{1,1}(b,R) \sim \osc^{1,1}_N(b,R)$.
\end{rem}

\begin{lem}\label{lem:truncosc}
Let $K_i$ be a non-degenerate kernel on $\R^{d_i}$, $i = 1,2$.
If we fix $A\geq 3$ in Definition \ref{def:tildeI} large enough (depending on the dimensions and kernel constants),
then for each $b \in L^1_{\loc}$, $1 \le v_i \le \infty$, $N \in \N$ and rectangle $R=I_1\times I_2 \subset\R^d$ we have
\begin{equation*}
\begin{split}
   \osc^{v_1,v_2}_N(b,R)
   \lesssim &\abs{\pair{[T_1,[b,T_2]] h_{0}}{1_{R_3} }} +  \abs{\pair{[T_1^*,[b,T_2]] h_{1}}{1_{R_2} }} \\
&+ \abs{\pair{[T_1,[b,T_2^*]] h_{2}}{1_{R_1} }} + \abs{\pair{[T_1^*,[b,T_2^*]] h_{3}}{1_{R_0} }}
\end{split}
\end{equation*}
for some functions $h_j$ with $\Norm{h_{j}}{L^{v_1'}L^{v_2'}(R_j)}\lesssim 1$ and $\Norm{h_{j}}{L^{\infty}(R_j)}\lesssim N$, where
$R_0=R$, $R_1=\wt I_1\times I_2$, $R_2=I_1\times\wt I_2$, $R_3=\wt R$.

If $K_1$ is \emph{symmetrically} non-degenerate, we get
\begin{equation*}
\begin{split}
   \osc^{v_1,v_2}_N(b,R)
   \lesssim &\abs{\pair{[T_1,[b,T_2]] h_{0}}{1_{R_3} }} +  \abs{\pair{[T_1,[b,T_2]] h_{1}}{1_{R_2} }} \\
&+ \abs{\pair{[T_1^*,[b,T_2^*]] h_{2}}{1_{R_1} }} + \abs{\pair{[T_1^*,[b,T_2^*]] h_{3}}{1_{R_0} }}.
\end{split}
\end{equation*}
A symmetric estimate holds if $K_2$ is \emph{symmetrically} non-degenerate.
If both $K_1$ and $K_2$
are symmetrically non-degenerate, then
\begin{equation}\label{eq:symdom}
   \osc^{v_1,v_2}_N(b,R)
   \lesssim\sum_{j=0}^3 \abs{\pair{[T_1,[b,T_2]]h_j}{1_{\wt R_j}}}.
\end{equation}
\end{lem}
\begin{proof}
Let $N$ be fixed.
Fix $f\in L^{v_1'}_0 L^{v_2'}_0(R)$ so that $\Norm{f}{L^{v_1'}_{0} L^{v_2'}_{0} (R)}\le 1$, $\|f\|_{L^{\infty}} \le N$
and $\osc^{v_1,v_2}_N(b,R) \lesssim |\pair{b}{f}|$.
We expand this $f$ according to Lemma \ref{lem:IteratedFactorization} to arrive at
\begin{equation*}
\begin{split}
  \pair{b}{f}
  &=\Bpair{b}{1_{\wt R} T_1T_2 h_0-T_2 h_0 \cdot T_1^*1_{\wt R}-T_1h_0\cdot T_2^*1_{\wt R}+h_0 T_1^*T_2^*1_{\wt R}} + \sum_{j=1}^3\pair{b}{\tilde{f}_j} \\
  &=\Bpair{b T_1T_2 h_0- T_1(b T_2 h_0)-T_2(bT_1 h_0)+ T_1T_2(bh_0)}{1_{\wt R} } + \sum_{j=1}^3\pair{b}{\tilde{f}_j} \\
  &=-\pair{[T_1,[b,T_2]] h_0}{1_{\wt R} } + \sum_{j=1}^3\pair{b}{\tilde{f}_j}.
\end{split}
\end{equation*}
Hence, we have
\begin{equation*}
\osc^{v_1,v_2}_N(b,R) \lesssim
  \abs{\pair{b}{f}}
  \leq\abs{\pair{[T_1,[b,T_2]] h_0}{1_{\wt R} }}+\sum_{j=1}^3\abs{\pair{b}{\tilde{f}_j}},
\end{equation*}
where, by Lemma \ref{lem:IteratedFactorization},
\begin{equation*}
  \abs{\pair{b}{\tilde{f}_j}}
  \lesssim \osc^{v_1,v_2}_N(b,R_j)A^{-\min(\alpha_1,\alpha_2)},
\end{equation*}
and thus
\begin{equation}\label{eq:osc1}
  \osc^{v_1,v_2}_N(b,R_0)\lesssim \abs{\pair{[T_1,[b,T_2]] h_0}{1_{R_3} }}+A^{-\min(\alpha_1,\alpha_2)}\sum_{j \in \{1,2,3\}}\osc^{v_1,v_2}_N(b,R_j)
\end{equation}
where $\Norm{h_0}{L^{v_1'}L^{v_2'}(R)}\lesssim A^d$ and $\Norm{h_0}{L^{\infty}(R)}\lesssim A^dN$.

Notice that if $\wt Q$ is a reflected cube of $Q$ with respect to a non-degenerate kernel $K$, then $Q$ is a reflected cube of $\wt Q$ with respect to $K^*$.
Therefore, we also get
\begin{equation}\label{eq:osc2}
  \osc^{v_1,v_2}_N(b,R_1)\lesssim \abs{\pair{[T_1^*,[b,T_2]] h_1}{1_{R_2} }}+A^{-\min(\alpha_1,\alpha_2)}\sum_{j \in \{0, 2,3\}} \osc^{v_1,v_2}_N(b,R_j),
\end{equation}
\begin{equation}\label{eq:osc3}
  \osc^{v_1,v_2}_N(b,R_2)\lesssim \abs{\pair{[T_1,[b,T_2^*]] h_2}{1_{R_1} }}+A^{-\min(\alpha_1,\alpha_2)}\sum_{j \in \{0, 1,3\}} \osc^{v_1,v_2}_N(b,R_j)
\end{equation}
and
\begin{equation}\label{eq:osc4}
  \osc^{v_1,v_2}_N(b,R_3)\lesssim \abs{\pair{[T_1^*,[b,T_2^*]] h_3}{1_{R_0} }}+A^{-\min(\alpha_1,\alpha_2)}\sum_{j \in \{0, 1,2\}} \osc^{v_1,v_2}_N(b,R_j).
\end{equation}
Combining \eqref{eq:osc1}, \eqref{eq:osc2}, \eqref{eq:osc3} and \eqref{eq:osc4}, and fixing $A$ large enough, we obtain by an absorption argument that
\begin{align*}
\osc^{v_1,v_2}_N(b,R)& \le \sum_{j=0}^3 \osc^{v_1,v_2}_N(b,R_j) \\
&\lesssim
\abs{\pair{[T_1,[b,T_2]] h_0}{1_{R_3} }} +  \abs{\pair{[T_1^*,[b,T_2]] h_1}{1_{R_2} }} \\
&+ \abs{\pair{[T_1,[b,T_2^*]] h_2}{1_{R_1} }} + \abs{\pair{[T_1^*,[b,T_2^*]] h_3}{1_{R_0} }},
\end{align*}
where $\Norm{h_j}{L^{v_1'}L^{v_2'}(R_j)}\lesssim 1$ and $\Norm{h_j}{L^{\infty}(R_j)}\lesssim N$.

Next, if $K$ is \emph{symmetrically} non-degenerate and $\wt Q$ is a reflected cube of $Q$ with respect to $K$, then $Q$ is a reflected cube of $\wt Q$ with respect to $K$
and $\wt Q$ is also a reflected cube of $Q$ with respect to $K^*$.
Suppose $K_1$ is symmetrically non-degenerate. Then we can replace \eqref{eq:osc2}
by
\begin{equation}\label{eq:osc2SYM}
  \osc^{v_1,v_2}_N(b,R_1)\lesssim \abs{\pair{[T_1,[b,T_2]] h_1}{1_{R_2} }}+A^{-\min(\alpha_1,\alpha_2)}\sum_{j \in \{0, 2,3\}} \osc^{v_1,v_2}_N(b,R_j),
\end{equation}
and 
\eqref{eq:osc3} by
\begin{equation}\label{eq:osc3SYM}
  \osc^{v_1,v_2}_N(b,R_2)\lesssim \abs{\pair{[T_1^*,[b,T_2^*]] h_2}{1_{R_1} }}+A^{-\min(\alpha_1,\alpha_2)}\sum_{j \in \{0, 1,3\}} \osc^{v_1,v_2}_N(b,R_j)
\end{equation}
The other symmetry statements are similar. This ends our proof.
\end{proof}
From this point on, given non-degenerate kernels $K_i$, we always consider $A$ fixed large enough so that the conclusion of Lemma \ref{lem:truncosc} hold.
\begin{rem}
We have that $[T_1, [b,T_2]]^* = [T_1^*, [b,T_2^*]]$, but a partial adjoint like $[T_1^*, [b,T_2]]$ can more easily be unbounded even if
$[T_1, [b,T_2]]$ is bounded. Thus, it can be useful to be able to eliminate these partial adjoint terms if at least one of the kernels
is symmetrically non-degenerate.
\end{rem}

\subsection{Off-support constants}
Notice that if we take a pairing, such as $\abs{\pair{[T_1,[b,T_2]] h_{0}}{1_{\tilde R} }}$,  from the conclusion of Lemma \ref{lem:truncosc},
then $h_0$ is supported on $I_1 \times I_2$ and $\wt R = \wt I_1 \times \wt I_2$ satisfies $\dist(I_i,\wt I_i) \sim \ell(I_i)$. 
Moreover, the case $v_1 = v_2 = 1$ is a bit special, as then $\|h_j\|_{L^{\infty}(R_j)} \lesssim 1$.
Otherwise, we only have the estimate $\Norm{h_{j}}{L^{v_1'}L^{v_2'}(R_j)} \lesssim 1$ at our disposal. 

We fix $b \in L^1_{\loc}$ and non-degenerate kernels $K_i$ -- the off-support constants will depend on these but this will not be denoted.
We first define off-support constants that are useful in the $v_1 = v_2 = 1$ situation -- which is a case that is encountered often.
For $s_1, s_2, t_1, t_2 \in (1,\infty)$ and $B(x,y) = b(x_1, x_2)-b(x_1, y_2)-b(y_1,x_2)+b(y_1, y_2)$ define
\begin{align*}
\Off_{(s_1, s_2)}^{(t_1, t_2)} =  \sup \frac{1}{|J_1|^{1+1/s_1-1/t_1}}   & \frac{1}{|J_2|^{1+1/s_2-1/t_2}}  \\
& \Big| \iint_{\R^{d}\times \R^d} B(x,y)
K_1(x_1, y_1) K_2(x_2, y_2)f_1(y)f_2(x) \ud y \ud x \Big|,
\end{align*}
where the supremum is taken over rectangles $P_1 = J_1 \times J_2$ and $P_2 = L_1 \times L_2$ with
$$
\ell(J_i) = \ell(L_i) \qquad \textup{and} \qquad
\dist (J_i, L_i)\sim \ell(J_i)
$$
and over functions $f_i \in L^{\infty}(P_i)$ with $\|f_j\|_{L^{\infty}} \le 1$. We also denote
$\Off_{(s_1, s_2)}^{(t_1, t_2), i*}$ the constant where $K_i$ is replaced by $K_i^*$.

If $T_i$ is a non-degenerate CZO with kernel $K_i$, then
$$\Off_{(s_1, s_2)}^{(t_1, t_2)} \le \| [T_1, [b,T_2]]\|_{L^{s_1}_{x_1}L^{s_2}_{x_2} \to L^{t_1}_{x_1}L^{t_2}_{x_2}}$$ whenever the commutator
exists as a bounded linear mapping. This explains why these weaker off-support constants make sense.

\begin{thm}\label{thm:mixednormrec}
Let $K_i$ be a non-degenerate kernel on $\R^{d_i}$, $i = 1,2$, and let $b \in L^1_{\loc}$. Let $p_1, p_2, q_1, q_2 \in (1,\infty)$. Then we have
for all rectangles $R=I_1\times I_2$ that
\begin{align*}
\frac{ \osc^{1,1}(b,R) }{|R|}
&\lesssim  \Big( \Off_{(p_1, p_2)}^{(q_1, q_2)} +  \Off_{(s_1, s_2)}^{(t_1, t_2), 1*}\Big) |I_1|^{\frac 1{p_1}-\frac 1{q_1}}|I_2|^{\frac 1{p_2}-\frac 1{q_2}}
\end{align*}
with any $s_i, t_i \in (1,\infty)$ satisfying $1/p_i - 1/q_i = 1/s_i - 1/t_i$, $i=1,2$.

If, in addition, one of the kernels $K_i$ is symmetrically non-degenerate,
then
$$
\frac{ \osc^{1,1}(b,R) }{|R|}
\lesssim \Off_{(p_1, p_2)}^{(q_1, q_2)}  |I_1|^{\frac 1{p_1}-\frac 1{q_1}}|I_2|^{\frac 1{p_2}-\frac 1{q_2}}.
$$
\end{thm}
\begin{proof}
By Lemma \ref{lem:truncosc} we have for each rectangle $R=I_1\times I_2 \subset\R^d$ that
\begin{equation*}
\begin{split}
   \osc^{1,1}(b,R)
   \lesssim &\abs{\pair{[T_1,[b,T_2]] h_{0}}{1_{R_3} }} +  \abs{\pair{[T_1^*,[b,T_2]] h_{1}}{1_{R_2} }} \\
&+ \abs{\pair{[T_1,[b,T_2^*]] h_{2}}{1_{R_1} }} + \abs{\pair{[T_1^*,[b,T_2^*]] h_{3}}{1_{R_0} }}
\end{split}
\end{equation*}
for some functions $h_j$ with $\Norm{h_{j}}{L^{\infty}(R_j)}\lesssim 1$, where
$R_0=R$, $R_1=\wt I_1\times I_2$, $R_2=I_1\times\wt I_2$, $R_3=\wt R$.
Clearly, we have
$$
\abs{\pair{[T_1,[b,T_2]] h_{0}}{1_{R_3} }} + \abs{\pair{[T_1^*,[b,T_2^*]] h_{3}}{1_{R_0} }} \lesssim  \Off_{(p_1, p_2)}^{(q_1, q_2)} |R| |I_1|^{\frac 1{p_1}-\frac 1{q_1}}|I_2|^{\frac 1{p_2}-\frac 1{q_2}}
$$
and
$$
\abs{\pair{[T_1^*,[b,T_2]] h_{1}}{1_{R_2} }} + \abs{\pair{[T_1,[b,T_2^*]] h_{2}}{1_{R_1} }} \lesssim \Off_{(s_1, s_2)}^{(t_1, t_2), 1*} |R| |I_1|^{\frac 1{p_1}-\frac 1{q_1}}|I_2|^{\frac 1{p_2}-\frac 1{q_2}}.
$$

The symmetry claim follows from the symmetry claims of Lemma \ref{lem:truncosc}.
\end{proof}

We now define larger off-support constants that can be used to control $\osc^{v_1, v_2}_{N}(b,R)$, uniformly on $N$, for general $v_1, v_2$.
For $s_1, s_2, t_1, t_2 \in (1,\infty)$ define
\begin{align*}
\wt{\Off}_{(s_1, s_2)}^{(t_1, t_2)} =  \sup \Big| \iint_{\R^{d}\times \R^d} B(x,y)
K_1(x_1, y_1) K_2(x_2, y_2)f_1(y)f_2(x) \ud y \ud x \Big|,
\end{align*}
where the supremum is taken over rectangles $P_1 = J_1 \times J_2$ and $P_2 = L_1 \times L_2$ with
$$
\ell(J_i) = \ell(L_i) \qquad \textup{and} \qquad
\dist (J_i, L_i)\sim \ell(J_i)
$$
and over functions $f_j \in L^{\infty}(P_j)$ with
$$
\|f_1\|_{L^{s_1}_{x_1}L^{s_2}_{x_2}} \le 1 \qquad \textup{and} \qquad  \|f_2\|_{L^{t_1'}_{x_1}L^{t_2'}_{x_2}} \le 1.
$$
We also define the dual off-support constants $\wt{\Off}_{(s_1, s_2)}^{(t_1, t_2), 1*}$, $\wt{\Off}_{(s_1, s_2)}^{(t_1, t_2), 2*}$ and
$\wt{\Off}_{(s_1, s_2)}^{(t_1, t_2), *}$ in the natural way.
Notice that $\Off_{(s_1, s_2)}^{(t_1, t_2)} \le \wt \Off_{(s_1, s_2)}^{(t_1, t_2)}$, and if $T_i$ is a non-degenerate CZO with kernel $K_i$, then
$\wt \Off_{(s_1, s_2)}^{(t_1, t_2)} \le \| [T_1, [b,T_2]]\|_{L^{s_1}_{x_1}L^{s_2}_{x_2} \to L^{t_1}_{x_1}L^{t_2}_{x_2}}$ whenever the commutator
exists as a bounded linear mapping.
Notice also that
$\wt{\Off}_{(s_1, s_2)}^{(t_1, t_2), *} = \wt{\Off}_{(t_1', t_2')}^{(s_1', s_2')}$.
\begin{thm}\label{thm:mixednormrecV2}
Let $K_i$ be a non-degenerate kernel on $\R^{d_i}$, $i = 1,2$, and let $b \in L^1_{\loc}$. Let $p_1, p_2, q_1, q_2 \in (1,\infty)$. Then we have
for all rectangles $R=I_1\times I_2$ that
\begin{align*}
\frac{\osc^{p_1', p_2'}(b,R)}{|I_1|^{1/p_1'}|I_2|^{1/p_2'}}
\lesssim \Big(\wt \Off_{(p_1, p_2)}^{(q_1, q_2)}+\wt \Off_{(p_1, p_2)}^{(q_1, q_2), 1*}
 + \wt \Off_{(p_1, p_2)}^{(q_1, q_2), 2*}+ \wt \Off_{(q_1', q_2')}^{(p_1', p_2')}
\Big) |I_1|^{\frac 1{p_1}-\frac 1{q_1}}|I_2|^{\frac 1{p_2}-\frac 1{q_2}}.
\end{align*}
If one of the kernels $K_i$ is symmetrically non-degenerate, we can drop the terms  $\wt \Off_{(p_1, p_2)}^{(q_1, q_2), 1*}$
and $\wt \Off_{(p_1, p_2)}^{(q_1, q_2), 2*}$, and if
both $K_1$ and $K_2$ are symmetrically non-degenerate, then
$$
\frac{\osc^{p_1', p_2'}(b,R)}{|I_1|^{1/p_1'}|I_2|^{1/p_2'}}
\lesssim \wt \Off_{(p_1, p_2)}^{(q_1, q_2)} |I_1|^{\frac 1{p_1}-\frac 1{q_1}}|I_2|^{\frac 1{p_2}-\frac 1{q_2}}.
$$
\end{thm}
\begin{proof}
Follows directly from Lemma \ref{lem:truncosc} similarly to Theorem \ref{thm:mixednormrec}.
\end{proof}

\section{Necessary conditions for commutator boundedness}\label{sec:nes}
For many of the commutator lower bounds (necessary conditions) we have already done the main work; oscillatory chararacterizations of function spaces
from Section \ref{sec:OscCharofSpaces} combined with the control of oscillations by off-support constants, Theorem \ref{thm:mixednormrec} and Theorem \ref{thm:mixednormrecV2}, will yield the results.
However, when there is an $\dot L^r$ norm involved, we will still have to use Lemma \ref{lem:truncosc} directly, and we will also need to introduce
yet another off-support constant.

Many of our necessary conditions are optimal; we will later
be able to prove the sufficiency of the obtained necessary condition, see Theorem \ref{thm:main}. A reason, but not the only reason, a lower bound
might not be optimal is that
the techniques used to obtain the lower bounds cannot distinguish between symmetric cases like
$$
p_1 < q_1 \textup{ and } p_2 = q_2
$$
and
$$
p_1 = q_1 \textup{ and } p_2 < q_2,
$$
and yield symmetric lower bounds: in the first case we can bound $\|b\|_{\dot C^{0, \beta_1}_{x_1}(\BMO_{x_2})}$ and in the second case $\|b\|_{\dot C^{0, \beta_2}_{x_2}(\BMO_{x_1})}$
by $\|[T_1, [b,T_2]]\|_{L^{p_1}_{x_1} L^{p_2}_{x_2} \to L^{q_1}_{x_1} L^{q_2}_{x_2}}$. Somewhat strikingly, both of these conditions happen to be optimal.
However, e.g. in the upper right corner of the table of Theorem \ref{thm:main}, where $p_1 > q_1$ and $p_2 < q_2$, the conditions do not match, while
in the symmetric case $p_1 < q_1$ and $p_2 > q_2$ (the lower left corner) they do match.

To have clean statements and to avoid repetition, we will
state all of the commutator lower bounds for \emph{symmetrically} non-degenerate kernels $K_i$. From the results of Section \ref{sec:WF}
it is clear that this is not necessary -- the general case just requires various dual off-support constants, and, as we saw, often it is enough that only
one of the kernels is symmetrically non-degenerate to get clean statements.

Sometimes we only obtain a necessary condition in terms of an oscillatory condition that we cannot relate to a function space norm. These
are recorded in this section but are not visible in the table of Theorem \ref{thm:main} (the three cases on the lower right corner).

In what follows we always have $p_1, p_2, q_1, q_2 \in (1,\infty)$, and
\begin{equation*}
\begin{split}
  \beta_i :=d_i\Big(\frac{1}{p_i}-\frac{1}{q_i}\Big),\quad \text{if}\quad p_i<q_i;\qquad
  \frac{1}{r_i} :=\frac{1}{q_i}-\frac{1}{p_i},\quad \text{if}\quad p_i>q_i.
\end{split}
\end{equation*}
Moreover, we also let $b \in L^1_{\loc}$ and $K_i$ be a symmetrically non-degenerate kernel on $\R^{d_i}$, $i = 1,2$. This data
is not repeated in the statements of the results.

\subsection{The condition $\dot C^{0,\beta_1}(\dot C^{0,\beta_2})$}
\begin{prop}\label{prop:cbeta12}
If $p_1 < q_1$ and $p_2 < q_2$, then
$$
\|b\|_{\dot C^{0,\beta_1}(\dot C^{0,\beta_2})}
\lesssim \Off_{(p_1, p_2)}^{(q_1, q_2)}.
$$
\end{prop}
\begin{proof}
By Proposition \ref{prop:calphacbeta} we have
$$
\|b\|_{\dot C^{0,\beta_1}(\dot C^{0,\beta_2})}
  \sim\sup\Big\{ \frac{1}{\ell(I_1)^{\beta_1}\ell(I_2)^{\beta_2}}\frac{\osc^{1,1}(b, I_1 \times I_2)}{|I_1| |I_2|}:
      I_i \subset\R^{d_i} \textup{ is a cube}\Big\},
$$
where, by Theorem \ref{thm:mixednormrec},
\begin{align*}
\frac{1}{\ell(I_1)^{\beta_1}\ell(I_2)^{\beta_2}}\frac{\osc^{1,1}(b, I_1 \times I_2)}{|I_1| |I_2|}
\lesssim \frac{1}{\ell(I_1)^{\beta_1}\ell(I_2)^{\beta_2}} \Off_{(p_1, p_2)}^{(q_1, q_2)}  |I_1|^{\frac 1{p_1}-\frac 1{q_1}}|I_2|^{\frac 1{p_2}-\frac 1{q_2}} 
=  \Off_{(p_1, p_2)}^{(q_1, q_2)}.
\end{align*}
\end{proof}

\subsection{The condition $\dot C^{0,\beta}(\BMO)$}

\begin{prop}\label{prop:CBMO}
If $p_1 < q_1$ and $p_2 = q_2$, then
$$
\|b\|_{\dot C^{0, \beta_1}_{x_1}(\BMO_{x_2})} \lesssim \Off_{(p_1, p_2)}^{(q_1, p_2)}.
$$
If $p_1 = q_1$ and $p_2 < q_2$, then
$$
\|b\|_{\dot C^{0, \beta_2}_{x_2}(\BMO_{x_1})} \lesssim \Off_{(p_1, p_2)}^{(p_1, q_2)}.
$$
\end{prop}
\begin{proof}
Let $p_1 < q_1$ and $p_2 = q_2$. By Proposition \ref{prop:CBMOchar} we have
$$
  \Norm{b}{\dot C^{0,\beta_1}_{x_1}(\BMO_{x_2})}
  \sim\sup\Big\{ \frac{1}{\ell(I_1)^{\beta_1}}\frac{\osc^{1,1}(b, I_1 \times I_2)}{|I_1| |I_2|}:
     I_i \subset\R^{d_i} \textup{ is a cube}\Big\},
$$
where, by Theorem \ref{thm:mixednormrec},
\begin{align*}
\frac{1}{\ell(I_1)^{\beta_1}}\frac{\osc^{1,1}(b, I_1 \times I_2)}{|I_1| |I_2|} 
\lesssim \frac{1}{\ell(I_1)^{\beta_1}} \Off_{(p_1, p_2)}^{(q_1, p_2)}  |I_1|^{\frac 1{p_1}-\frac 1{q_1}} = \Off_{(p_1, p_2)}^{(q_1, p_2)}.
\end{align*}
The proof of the case $p_1 = q_1$ and $p_2 < q_2$ is symmetric.
\end{proof}

\subsection{The condition $\BMO_{\operatorname{rect}}$}
We record the following lower bound concerning the diagonal of Theorem \ref{thm:main}, even though our real focus is on the off-diagonal cases.
\begin{prop}\label{prop:RectBMO}
If $p_1 = q_1$ and $p_2 = q_2$ we have
$$
\|b\|_{\BMO_{\textup{rect}}(p_1', p_2')}
\lesssim \wt \Off_{(p_1, p_2)}^{(p_1, p_2)}
$$
\end{prop}
\begin{proof}
By definition and Theorem \ref{thm:mixednormrecV2} we have
$$
\|b\|_{\BMO_{\textup{rect}}(p_1', p_2')} = \sup_{R=I_1 \times I_2} \frac{\osc^{p_1', p_2'}(b, R)}{|I_1|^{1/{p_1'}} |I_2|^{1/{p_2'}}} \lesssim
\wt \Off_{(p_1, p_2)}^{(p_1, p_2)}.
$$
\end{proof}

\subsection{Conditions related to $\dot C^{0,\beta}(\dot L^r)$ and $\BMO(\dot L^r)$}
We need a new off-support constant to deal with the calculations arising from considerations involving the space $\dot L^r$. 
Define $B(x,y) = b(x_1, x_2)-b(x_1, y_2)-b(y_1,x_2)+b(y_1, y_2)$ and set
\begin{align*}
\Off_{(p_1, p_2), \Sigma}^{(q_1, q_2)} =  \sup \frac{\sum_{i=1}^N  
 \Big| \iint_{\R^{d}\times \R^d} B(x,y)
 K_1(x_1, y_1) K_2(x_2, y_2)f_{1,i}(y)f_{2,i}(x) \ud y \ud x \Big|}{
 \Big\| \sum_{i=1}^N  \|f_{1,i}\|_{L^{\infty}}1_{P_{1,i}} \|_{L^{p_1}_{x_1}L^{p_2}_{x_2}}
 \Big\| \sum_{i=1}^N  \|f_{2,i}\|_{L^{\infty}}1_{P_{2,i}} \|_{L^{q_1'}_{x_1}L^{q_2'}_{x_2}}
  },
\end{align*}
where the supremum is taken over rectangles $P_{1,i} = J_{1,i} \times J_{2,i}$ and $P_{2,i} = L_{1, i} \times L_{2,i}$ with
$$
\ell(J_{k,i}) = \ell(L_{k,i}) \qquad \textup{and} \qquad
\dist (J_{k,i}, L_{k,i})\sim \ell(J_{k,i})
$$
and over functions $f_{k,i} \in L^{\infty}(P_{k,i})$, $k=1,2$, $i=1,\ldots, N$.
Using that for linear operators $U$ we have the random sign trick
$$
\sum_{i=1}^N \langle Uf_i, g_i \rangle = \E \Big\langle U\Big( \sum_{i=1}^N \eps_i f_i \Big),  \sum_{j=1}^N \eps_j g_j \Big\rangle,
$$
we see that $\Off_{(p_1, p_2), \Sigma}^{(q_1, q_2)} \le \| [T_1, [b,T_2]]\|_{L^{p_1}_{x_1}L^{p_2}_{x_2} \to L^{q_1}_{x_1}L^{q_2}_{x_2}}$ whenever
the commutator exists as a bounded operator. Thus, this is a reasonable off-support constant.
\begin{lem}\label{lem1}
Suppose $p_2 > q_2$. For
all sparse collections $\mathscr S_2$ of cubes in $\R^{d_2}$, non-negative coefficients $\lambda_{S_2}$ satisfying $\sum_{S_2\in\mathscr S_2}\lambda_{S_2}^{r_2'}\abs{S_2}\leq 1$
and cubes $I_1 \subset \R^{d_1}$, we have
\begin{equation*}
  \sum_{S_2\in\mathscr S_2} \lambda_{S_2}\osc^{1,1}(b,I_1\times S_2)
  \lesssim \Off_{(p_1, p_2), \Sigma}^{(q_1, q_2)} \abs{I_1}^{1+(1/p_1-1/q_1)}.
\end{equation*}
\end{lem}

\begin{proof}
Using Lemma \ref{lem:truncosc} we estimate
$$
   \sum_{S_2\in\mathscr S_2} \lambda_{S_2}\osc^{1,1}(b,I_1\times S_2) 
   \lesssim \sum_{i=0}^3 \sum_{S_2\in\mathscr S_2}\lambda_{S_2} |\pair{[T_1,[b,T_2]]h_{(I_1\times S_2)_i}}{1_{(I_1\times S_2)_i^{\sim}}}|
$$
for some $h_{(I_1\times S_2)_i} \in L^{\infty}((I_1\times S_2)_i)$ with $\|h_{(I_1\times S_2)_i}\|_{L^{\infty}} \lesssim 1$. 
Using the algebra
$$
\frac{r_2'}{p_2} + \frac{r_2'}{q_2'} = r_2'\Big( \frac{1}{p_2} + 1 - \frac{1}{q_2}\Big) 
= r_2'\Big(1 - \frac{1}{r_2}\Big) =1
$$
we write this in the form
$$
\sum_{i=0}^3 \sum_{S_2\in\mathscr S_2} |\pair{[T_1,[b,T_2]](\lambda_{S_2}^{r_2'/p_2}h_{(I_1\times S_2)_i})}{\lambda_{S_2}^{r_2'/q_2'}1_{(I_1\times S_2)_i^{\sim}}}|.
$$
This can directly be dominated with
$$
\Off_{(p_1, p_2), \Sigma}^{(q_1, q_2)} \sum_{i=0}^3\Big\| \sum_{S_2\in\mathscr S_2} \lambda_{S_2}^{r_2'/p_2}1_{(I_1\times S_2)_i} \Big\|_{L^{p_1}_{x_1}L^{p_2}_{x_2}}
\Big\| \sum_{S_2\in\mathscr S_2} \lambda_{S_2}^{r_2'/q_2'}1_{(I_1\times S_2)_i^{\sim}} \Big\|_{L^{q_1'}_{x_1}L^{q_2'}_{x_2}}.
$$

Here $1_{(I_1\times S_2)_i}\leq 1_{I_1^*\times S_2^*}$, where $I_1^*\supset I_1\cup\wt I_1$ and $S_2^*\supset S_2\cup\wt S_2$ are concentric dilations of $I_1$ and $S_2$ by a bounded factor. Since the collection $\{S_2^*:S_2\in\mathscr S_2\}$ is still sparse (with a different sparseness constant), it follows from \eqref{eq:basicSparse} 
and $\sum_{S_2\in\mathscr S_2}\lambda_{S_2}^{r_2'}\abs{S_2}\leq 1$ that
\begin{equation*}
\begin{split}
  &\Big\| \sum_{S_2\in\mathscr S_2} \lambda_{S_2}^{r_2'/p_2}1_{(I_1\times S_2)_i} \Big\|_{L^{p_1}_{x_1}L^{p_2}_{x_2}}
  \le \Big\|\sum_{S_2\in\mathscr S_2} \lambda_{S_2}^{r_2'/p_2}1_{I_1^* \times S_2^*} \Big\|_{L^{p_1}_{x_1}L^{p_2}_{x_2}} \\
  &=\Norm{1_{I_1^*}}{L^{p_1}}\BNorm{\sum_{S_2\in\mathscr S_2}\lambda_{S_2}^{r_2'/p_2}1_{S_2^*}}{L^{p_2}}
  \lesssim\abs{I_1}^{1/p_1}\Big(\sum_{S_2 \in\mathscr S_2}(\lambda_{S_2}^{r_2'/p_2})^{p_2}\abs{S_2}\Big)^{1/p_2} \le \abs{I_1}^{1/p_1}.
\end{split}
\end{equation*}
Similarly, using $1_{(I_1\times S_2)_i^{\sim}}\leq 1_{I^*_1\times S_2^*}$, we conclude that
\begin{equation*}
  \Big\| \sum_{S_2\in\mathscr S_2} \lambda_{S_2}^{r_2'/q_2'}1_{(I_1\times S_2)_i^{\sim}} \Big\|_{L^{q_1'}_{x_1}L^{q_2'}_{x_2}}
  \lesssim\abs{I_1}^{1/q_1'},
\end{equation*}
and this ends the proof.
\end{proof}

The symmetric version reads as follows.
\begin{lem}
Suppose $p_1 > q_1$. For
all sparse collections $\mathscr S_1$ of cubes in $\R^{d_1}$, non-negative coefficients $\lambda_{S_1}$ satisfying $\sum_{S_1\in\mathscr S_1}\lambda_{S_1}^{r_1'}\abs{S_1}\leq 1$
and cubes $I_2 \subset \R^{d_2}$, we have
\begin{equation*}
  \sum_{S_1\in\mathscr S_1} \lambda_{S_1}\osc^{1,1}(b,S_1\times I_2)
  \lesssim \Off_{(p_1, p_2), \Sigma}^{(q_1, q_2)} \abs{I_2}^{1+(1/p_2-1/q_2)}.
\end{equation*}
\end{lem}

Relating the oscillatory constants to function spaces we can say the following.
\begin{prop}\label{prop:CLrBMOLr}
If $p_1 < q_1$ and $p_2 > q_2$ then
$$
\|b\|_{\dot C^{0,\beta_1}_{x_1}(\dot L^{r_2}_{x_2})} \lesssim \Off_{(p_1, p_2), \Sigma}^{(q_1, q_2)}
$$
and if $p_1 = q_1$ and $p_2 > q_2$ then
\begin{align*}
\sup\Big\{ \frac{1}{\abs{I_1}} \sum_{S_2 \in\mathscr S_2}\lambda_{S_2} \osc^{1,1}(b, I_1 \times S_2) \colon 
     I_1 \subset\R^{d_1},\mathscr S_2 \text{ sparse},&\sum_{S_2 \in\mathscr S_2 }\abs{S_2}\lambda_{S_2}^{r_2'}\leq 1\Big\} \\
     &\lesssim \Off_{(p_1, p_2), \Sigma}^{(p_1, q_2)},
\end{align*}
where the supremum on the left hand side is also dominated by $\|b\|_{\BMO_{x_1}(\dot L^{r_2}_{x_2})}$. The symmetric statements
hold if $p_1 > q_1$ and $p_2 < q_2$ or $p_1 > q_1$ and $p_2=q_2$.
\end{prop}
\begin{proof}
Follows by combining the above two lemmas with Proposition \ref{prop:CLr} and Lemma \ref{lem:BMOLr}.
\end{proof}

\subsection{A condition related to $\dot L^{r_1}(\dot L^{r_2})$}

\begin{prop}\label{prop:dotdot}
If $p_1>q_1$ and $p_2>q_2$, then
for all sparse collections $\mathscr S_i$ of cubes in $\R^{d_i}$ and non-negative coefficients $\lambda_{i, S_i}$ satisfying $\sum_{S_i\in\mathscr S_i}\lambda_{i, S_i}^{r_i'}|S_i| \le 1$
we have
\begin{equation*}
  \sum_{S_1 \in\mathscr S_1}\sum_{S_2 \in\mathscr S_2} \lambda_{1,S_1}\lambda_{2,S_2}\osc^{1,1}(b,S_1\times S_2)
  \lesssim \Off_{(p_1, p_2), \Sigma}^{(q_1,q_2)}.
\end{equation*}
\end{prop}

\begin{proof}
Using Lemma \ref{lem:truncosc} we estimate
\begin{equation*}
\begin{split}
   \sum_{S_1\in\mathscr S_1}&\sum_{S_2 \in\mathscr S_2} \lambda_{1,S_1}\lambda_{2,S_2}\osc^{1,1}(b,S_1\times S_2) \\
   &\lesssim \sum_{i=0}^3 \sum_{S_1\in\mathscr S_1}\sum_{S_2\in\mathscr S_2} \lambda_{1,S_1}\lambda_{2,S_2} \abs{ \pair{[T_1,[b,T_2]]h_{(S_1\times S_2)_i}}{1_{(S_1\times S_2)_i^{\sim}}}  } \\
   &= \sum_{i=0}^3 \sum_{S_1\in\mathscr S_1}\sum_{S_2\in\mathscr S_2} \abs{ \pair{[T_1,[b,T_2]](\lambda_{1,S_1}^{r_1'/p_1}\lambda_{2,S_2}^{r_2'/p_2}h_{(S_1\times S_2)_i})}{\lambda_{1,S_1}^{r_1'/q_1'}\lambda_{2,S_2}^{r_2'/q_2'}1_{(S_1\times S_2)_i^{\sim}}} } \\
      &\leq \sum_{i=0}^3 \Off_{(p_1, p_2), \Sigma}^{(q_1,q_2)} 
   \BNorm{\sum_{\substack{ S_1\in\mathscr S_1\\ S_2\in\mathscr S_2}}\lambda_{1,S_1}^{r_1'/p_1}\lambda_{2,S_2}^{r_2'/p_2} 1_{(S_1\times S_2)_i}}{L^{p_1}_{x_1} L^{p_2}_{x_2}} \times \\
     &\hspace{5cm} \BNorm{ \sum_{\substack{ S_1 \in\mathscr S_1 \\ S_2 \in\mathscr S_2}} \lambda_{1,S_1}^{r_1'/q_1'}\lambda_{2,S_2}^{r_2'/q_2'}1_{(S_1\times S_2)_i^{\sim}}}{L^{q_1'}_{x_1}L^{q_2'}_{x_2}}.
\end{split}
\end{equation*}
Using the notation of the proof of Lemma \ref{lem1} we have
\begin{equation*}
\begin{split}
  \BNorm{\sum_{\substack{ S_1\in\mathscr S_1\\ S_2\in\mathscr S_2}}&\lambda_{1,S_1}^{r_1'/p_1}
  \lambda_{2,S_2}^{r_2'/p_2} 1_{(S_1\times S_2)_i}}{L^{p_1}_{x_1} L^{p_2}_{x_2}} \\
  &\leq \BNorm{\sum_{\substack{ S_1\in\mathscr S_1\\ S_2\in\mathscr S_2}}\lambda_{1,S_1}^{r_1'/p_1}
  \lambda_{2,S_2}^{r_2'/p_2} 1_{S_1^* \times S_2^*}}{L^{p_1}_{x_1} L^{p_2}_{x_2}} \\
  &=\BNorm{\sum_{S_1\in\mathscr S_1}\lambda_{1,S_1}^{r_1'/p_1} 1_{S_1^*}}{L^{p_1}_{x_1}} \BNorm{\sum_{S_2 \in\mathscr S_2}\lambda_{2,S_2}^{r_2'/p_2} 1_{S_2^*}}{L^{p_2}_{x_2}},
\end{split}
\end{equation*}
where for both $i=1,2$ we have by the sparseness of the family $\{S_i^*: S_i \in\mathscr S_i\}$ that
\begin{equation*}
  \BNorm{\sum_{S_i\in\mathscr S_i}\lambda_{i,S_i}^{r_i'/p_i} 1_{S_i^*}}{L^{p_i}_{x_i}}
  \lesssim\Big(\sum_{S_i\in\mathscr S_i}(\lambda_{i,S_i}^{r_i'/p_i})^{p_i}\abs{S_i}\Big)^{1/p_i}\leq 1,
\end{equation*}
and the the estimate for the $L^{q_1'}_{x_1}L^{q_2'}_{x_2}$ norm is completely analogous.
\end{proof}
Of course, we would like to be able to bound $\|b\|_{\dot L^{r_1}_{x_1} \dot L^{r_2}_{x_2}}$ above, but we only know the direction \eqref{eq:Lrr}.

\begin{rem}
For $p_1=p_2=:p$ and $q_1=q_2=:q$, almost the same argument gives the following slightly stronger conclusion: Suppose that $[T_1,[b,T_2]]: L^p \to L^q$, where $p>q$ (this can be replaced with the off-support constants as usual).
Then for $1/r=1/q-1/p$,
 any sparse collections $\mathscr R$ of {\em rectangles} $R=I_1\times I_2$ in $\R^{d_1}\times\R^{d_2}$ and coefficients $\sum_{R\in\mathscr R}\lambda_{R}^{r'}\abs{R}\leq 1$, we have
\begin{equation*}
  \sum_{R\in\mathscr R} \lambda_{R}\osc^{1,1}(b,R)
  \lesssim \Norm{[T_1,[b,T_2]]}{L^{p}\to L^{q}}.
\end{equation*}
The difference is that the coefficients are not of the product form, while above 
we had $\mathscr R=\mathscr S_1\times\mathscr S_2=\{S_1\times S_2:S_1\in\mathscr S_1,S_2\in\mathscr S_2\}$.
The point of the generalisation is that we can still split the coefficients $\lambda_R=\lambda_R^{r'/p}\lambda_R^{r'/q'}$ and estimate
\begin{equation*}
  \BNorm{\sum_{R\in\mathscr R}\lambda_R^{r'/p}1_{R^*}}{L^p}
  \lesssim\Big(\sum_{R\in\mathscr R}(\lambda_R^{r'/p})^p\abs{R}\Big)^{1/p}\lesssim 1;
\end{equation*}
in the case of mixed norms, it seems unclear how to perform this splitting of the coefficients. But even with the stronger condition, we do not know how to relate it to $b\in\dot L^r(\dot L^r)$. This problem is a variant of the question of bi-parameter sparse domination that has attracted some attention. Recently, \cite{BCOR} showed the {\em failure} of the natural sparse form bound for the strong maximal operator. However, our problem is more like finding a sparse bound for the {\em identity operator}!
\end{rem}

\section{Sufficient conditions for commutator boundedness}\label{sec:suf}
In this section we are given Calder\'on--Zygmund operators $T_i$ on $\R^{d_i}$ with standard kernels $K_i$, $i = 1,2$. We also always assume that $b\in L^s_{\loc}(\R^d)$ for some $s > 1$. 
Moreover, we let $p_1, p_2, q_1, q_2 \in (1,\infty)$ and use the familiar notation
\begin{equation*}
\begin{split}
  \beta_i :=d_i\Big(\frac{1}{p_i}-\frac{1}{q_i}\Big),\quad \text{if}\quad p_i<q_i;\qquad
  \frac{1}{r_i} :=\frac{1}{q_i}-\frac{1}{p_i},\quad \text{if}\quad p_i>q_i.
\end{split}
\end{equation*}
This fixed data is not repeated in the statements of the results. 

We aim to prove
$$
\| [T_1, [b, T_2]] \|_{L^{p_1}_{x_1}L^{p_2}_{x_2} \to L^{q_1}_{x_1}L^{q_2}_{x_2} } \lesssim \|b\|_{Y_{p_i, q_i}}
$$
preferably with a function space $Y_{p_i, q_i}$ that matches the corresponding necessary condition obtained in Section \ref{sec:nes}.

Along the way we will need several kernel representations for $\langle [T_1, [b, T_2]]f, g \rangle$ even if $f$ and $g$ do not have disjoint supports.
To make these rigorous we recall the notion of maximal truncations of CZOs.
For $\epsilon > 0$ we define for all $x_i \in \R^{d_i}$ and $f \in L^p(\R^{d_i})$, $1 \le p < \infty$, that
$$
T_{i, \epsilon} f(x_i) = \int_{|x_i-y_i| > \epsilon} K_i(x_i,y_i) f(y_i) \ud y_i.
$$
We define the maximal truncation of $T_i$ via the formula
$$
T_{i,*} f(x_i) = \sup_{\epsilon > 0} |T_{i, \epsilon} f(x_i)|.
$$
For the CZO $T_i$ the following Cotlar's inequality is true: for $0 < r < 1$ we have for all $x_i \in \R^{d_i}$ that
$$
T_{i,*} f(x_i) \lesssim_r M_r T_if(x_i) +  Mf(x_i).
$$
Here $M$ is the Hardy--Littlewood maximal function on $\R^{d_i}$ and $M_r g = (M |g|^r)^{1/s}$. The implicit
constant in Cotlar's inequality depends on $\|T_i\|_{L^2 \to L^2}$, but this is the type of CZO data that we will not track.
The consequence of Cotlar's inequality is that also $T_*$ maps $L^p(\R^{d_i})$ to $L^p(\R^{d_i})$ for all $p \in (1,\infty)$.
A standard fact also is that
\begin{equation}\label{eq:Tred}
T_i f = T_{i,0}f + a_i f,
\end{equation}
where $a_i$ is a bounded function on $\R^{d_i}$, $T_{i,0}$ is bounded on $L^p(\R^{d_i})$ for all $p \in (1,\infty)$, and for some sequence $\epsilon_j \to 0$ we have
$$
\int T_{i,0}f \cdot g   = \lim_{j \to \infty} \int T_{i, \epsilon_j} f \cdot g
$$
whenever $f \in L^p(\R^{d_i})$ and $g \in L^{p'}(\R^{d_i})$ for some $p \in (1,\infty)$.

For the kernel representations it is convenient to again denote
$$
B(x,y) = b(x_1, x_2)-b(x_1, y_2)-b(y_1,x_2)+b(y_1, y_2).
$$

\subsection{$\dot C^{0, \alpha}(X)$}
We begin with our first kernel representation.
\begin{prop}\label{prop:pointwise}
If
$$
(x, y) \mapsto B(x,y)
 K_1(x_1, y_1) K_2(x_2, y_2) \in L^{1}_{\loc}(\R^{2d}),
$$
then for all $f,g\in L_c^\infty(\R^d)$ we have
\begin{align*}
\bla [T_1, [b, T_2]]f, g\bra= - \iint_{\R^{d}\times \R^d} B(x,y)
 K_1(x_1, y_1) K_2(x_2, y_2)f(y)g(x) \ud y \ud x.
\end{align*}
\end{prop}
\begin{proof}
We write $T_i h = T_{i,0}h + a_i h$ as in \eqref{eq:Tred}.
Notice that $[T_1, [b, T_2]] = [T_{1,0}, [b, T_{2,0}]]$ and that we have by definition
\[
[T_{1,0}, [b, T_{2,0}]]f=T_{1,0} (bT_{2,0}f)+T_{2,0} (bT_{1,0}f)- bT_{1,0}T_{2,0}f-T_{1,0}T_{2,0} (bf).
\]
We consider these terms separately. For all $t \in (1,s)$ we have
for almost every $x_2$ that $x_1 \mapsto b(x_1, x_2)T_{2,0}(f(x_1, \cdot))(x_2) \in L^{t}(\R^{d_1})$. Therefore, we have
\[
\langle T_{1,0} (bT_{2,0}f), g\rangle_{\R^{d_1}} = \lim_{j\to \infty}\langle T_{1, \varepsilon_j} (bT_{2,0}f), g\rangle_{\R^{d_1}}.
\]
Since $T_{1,*}(bT_{2,0}f)g\in L^1(\R^d)$
we have by dominated convergence theorem and Fubini's theorem that
\begin{align*}
\langle T_{1,0} (bT_{2,0}f), g\rangle&= \lim_{j\to \infty}\langle T_{1, \varepsilon_j} (bT_{2,0}f), g\rangle = \lim_{j\to \infty}\int_{\R^{d_1}}\langle T_{2,0}f, bT_{1,\varepsilon_j}^* g\rangle_{\R^{d_2}}.
\end{align*}
Similarly as above, we further have
$$
\langle T_{1,0} (bT_{2,0}f), g\rangle= \lim_{j\to \infty} \int_{\R^{d_1}}\lim_{k \to \infty}\langle T_{2, \rho_k}f, bT_{1,\varepsilon_j}^* g\rangle_{\R^{d_2}}
=  \lim_{j\to \infty} \lim_{k \to \infty} \langle T_{1,\varepsilon_j} (bT_{2,\rho_k}f), g\rangle.
$$

We can similarly write out all the other terms of the commutator $\langle [T_{1,0}, [b, T_{2,0}]]f, g \rangle$.
Thus, we have
\begin{align*}
\langle [T_{1,0}, [b, &T_{2,0}]]f, g \rangle = \lim_{j\to \infty} \lim_{k \to \infty} \langle [T_{1,\varepsilon_j}, [b, T_{2,\rho_k}]]f, g \rangle \\
&=- \lim_{j\to \infty} \lim_{k \to \infty}\int_{\R^{d}}\int_{\substack{y_1 \colon  |y_1 - x_1| > \epsilon_j \\ y_2\colon |y_2 - x_2| > \rho_k}} B(x,y)
K_1(x_1, y_1) K_2(x_2, y_2)f(y)g(x) \ud y \ud x.
\end{align*}
The proof is finished by a final application of the dominated convergence theorem.
\end{proof}
We now give the upper bound that matches Proposition \ref{prop:cbeta12}.
\begin{prop}
If $p_1 < q_1$ and $p_2 < q_2$, we have
\[
\|[T_1, [b, T_2]]\|_{L^{p_1}_{x_1}L^{p_2}_{x_2}\to L^{q_1}_{x_1}L^{q_2}_{x_2}}\lesssim \|b\|_{\dot C^{0, \beta_1}(\dot C^{0,\beta_2})}.
\]
\end{prop}
\begin{proof}
By Proposition \ref{prop:pointwise} we have for $b \in \dot C^{0, \beta_1}(\dot C^{0,\beta_2})$ that
\begin{align*}
\big|\langle [T_1, [b, T_2]]f, g\rangle \big|\le \|b\|_{\dot C^{0, \beta_1}(\dot C^{0,\beta_2})}\langle I_{1,\beta_1}I_{2, \beta_2}|f|, |g|\rangle,
\end{align*}
where
\[
I_{1,\beta_1}f=\int_{\R^{d_1}}\frac{f(y_1)}{|x_1-y_1|^{d_1-\beta_1}}dy_1
\]is the fractional integral hitting on $\R^{d_1}$ and $I_{2,\beta_2}$ is defined similarly.  The proof is completed by using the well-known boundedness of the fractional integrals:
$$
I_{i, \beta_i} \colon L^{p_i} \to L^{q_i}.
$$
Indeed, notice that by Minkowski's integral inequality and the positivity of the operator
$$
\| I_{1, \beta_1} |h| \|_{L^{q_1}_{x_1} L^{q_2}_{x_2}} \le \| I_{1, \beta_1} \|h\|_{L^{q_2}_{x_2}} \|_{L^{q_1}_{x_1}}
\lesssim \| \|h\|_{L^{q_2}_{x_2}} \|_{L^{p_1}_{x_1}} = \|h\|_{L^{p_1}_{x_1} L^{q_2}_{x_2}},
$$
so that no vector-valued theory is really used.
\end{proof}

We move on to proving the remaining $\dot C^{0, \alpha}(X)$ type upper bounds. Again, a kernel representation is required.
\begin{prop}\label{prop:pointwiseV2}
If $b\in \dot C^{0, \alpha}_{x_1}(\BMO_{x_2})$ for some $\alpha > 0$,
then for all $f,g\in L_c^\infty(\R^d)$ we have
\begin{align*}
\bla [T_1, [b, T_2]]f, g\bra= - \iint_{\R^{d_1}\times \R^{d_1}} K_1(x_1, y_1) \big\langle  \big[ b(x_1)-b(y_1) , T_2\big]f(y_1), g(x_1) \big\rangle_{\R^{d_2}}.
\end{align*}
\end{prop}

\begin{proof}
We write
\begin{align*}
 \bla [T_1, [b, T_2]]f, g\bra  = \bla [T_{1,0}, [b, T_{2}]]f, g\bra& =\bla [T_{1,0}, [\wt b, T_{2}]]f, g\bra \\
 &=  \bla T_{2}[\wt b, T_{1,0}]f, g\bra- \bla [\wt b, T_{1,0}] T_{2} f, g\bra,
\end{align*}
where $\wt b= b- \langle b\rangle_{V, 2}$ and $V \subset \R^{d_2}$ is a cube with $\supp_{\R^{d_2}} f\cup \supp_{\R^{d_2}} g\subset V $.
We can write
\begin{align*}
\bla T_{2}[\wt b, T_{1,0}]f, g\bra&=\bla  T_{1,0}f, \wt bT_2^* g\bra- \bla  T_{1,0}(\wt bf), T_2^* g\bra\\&=
\int_{\R^{d_2}} \lim_{j\to \infty}\big[\bla  T_{1,\varepsilon_j}f, \wt bT_2^* g\bra_{\R^{d_1}}- \bla  T_{1,\varepsilon_j}(\wt bf), T_2^* g\bra_{\R^{d_1}}\big]\\
&=\lim_{j\to \infty}\int_{\R^{d_2}} \big[\bla  T_{1,\varepsilon_j}f, \wt bT_2^* g\bra_{\R^{d_1}}- \bla  T_{1,\varepsilon_j}(\wt bf), T_2^* g\bra_{\R^{d_1}}\big],
\end{align*}
where in the last step we have used the dominated convergence theorem. Notice that we can now write the above as
\begin{align*}
 \lim_{j\to \infty} \int_{\R^{d_2}} \int_{\R^{d_1}}\int_{y_1: |x_1-y_1|>\varepsilon_j} K_1(x_1, y_1) (\wt b(x_1, y_2)-\wt b(y_1, y_2)) f(y_1, y_2) T_2^* g(x_1, y_2)\ud y_1\ud x_1 \ud y_2.
\end{align*}
Since $b\in \dot C^{0, \alpha}_{x_1}(\BMO_{x_2})$, by definition we know that for all $p \in (1,\infty)$ we have
\[
y_2 \mapsto \frac{\wt b(x_1, y_2)-\wt b(y_1, y_2)}{|x_1-y_1|^{\alpha}}=\frac{b(x_1, y_2)-b(y_1, y_2)-\langle b(x_1, \cdot)-b(y_1, \cdot)\rangle_{V}}{|x_1-y_1|^{\alpha}}\in L^p(V)
\] uniformly for all $x_1, y_1$ with $x_1\neq y_1$. Using the above and the boundedness of the fractional integrals we can verify that
\[
(x_1, y_1, y_2) \mapsto K_1(x_1, y_1) (\wt b(x_1, y_2)-\wt b(y_1, y_2)) f(y_1, y_2) T_2^* g(x_1, y_2) \in L^1(\R^{2d_1+d_2}).
\]Therefore, we have
\begin{align*}
\bla &T_{2}[\wt b, T_{1,0}]f, g\bra\\&=  \int_{\R^{2d_1}}\int_{\R^{d_2}} K_1(x_1, y_1)T_2  ((\wt b(x_1, \cdot)-\wt b(y_1, \cdot)) f(y_1, \cdot ))(y_2) g(x_1, y_2) \ud y_2
\ud y_1\ud x_1.
\end{align*}
We complete the proof by noting that similarly $\bla [\wt b, T_{1,0}]T_2 f, g\bra$ equals
$$
 \int_{\R^{2d_1}}\int_{\R^{d_2}} K_1(x_1, y_1) (\wt b(x_1, y_2)-\wt b(y_1, y_2)) T_2 f(y_1, y_2) g(x_1, y_2) \ud y_2
\ud y_1\ud x_1.
$$
\end{proof}
We obtain the following counterpart of Proposition \ref{prop:CBMO}.
\begin{prop}
If $p_1 < q_1$ and $p_2 = q_2$, then
$$
\|[T_1, [b, T_2]]\|_{L^{p_1}_{x_1}L^{p_2}_{x_2}\to L^{q_1}_{x_1}L^{p_2}_{x_2}} \lesssim 
\|b\|_{\dot C^{0, \beta_1}_{x_1}(\BMO_{x_2})}.
$$
If $p_1 = q_1$ and $p_2 < q_2$, then
$$
\|[T_1, [b, T_2]]\|_{L^{p_1}_{x_1}L^{p_2}_{x_2}\to L^{p_1}_{x_1}L^{q_2}_{x_2}} \lesssim 
\|b\|_{\dot C^{0, \beta_2}_{x_2}(\BMO_{x_1})}.
$$
\end{prop}
\begin{proof}
Let first $p_1 < q_1$ and $p_2 = q_2$.
For $f, g\in L_c^\infty$ we have by Proposition \ref{prop:pointwiseV2} that
\begin{align*}
  &\big|\bla [T_1, [b, T_2]]f, g\bra\big|\\&= \left|\iint_{\R^{d_1}\times \R^{d_1}}|x_1-y_1|^{\beta_1}K_1(x_1, y_1) \Big\langle  \Big[\frac{b(x_1)-b(y_1)}{|x_1-y_1|^{\beta_1}}, T_2\Big]f(y_1), g(x_1)\Big\rangle_{\R^{d_2}}\right|\\
  &\lesssim \iint_{\R^{d_1}\times \R^{d_1}}|x_1-y_1|^{\beta_1-d_1}  \BNorm{ \frac{b(x_1)-b(y_1)}{|x_1-y_1|^{\beta_1}} }{\BMO_{x_2}} 
     \Norm{f(y_1)}{L_{x_2}^{p_2}} \Norm{g(x_1)}{L_{x_2}^{p_2'}}  \\
  &\lesssim \|b\|_{\dot C^{0, \beta_1}_{x_1}(\BMO_{x_2})} \|f\|_{L^{p_1}_{x_1}L^{p_2}_{x_2}}\|g\|_{L^{q_1'}_{x_1}L^{p_2'}_{x_2}},
\end{align*}
where we have used the well-known boundedness of one parameter commutators and the fractional integrals.

Next, let $p_1 = q_1$ and $p_2 < q_2$. Similarly as in Proposition \ref{prop:pointwiseV2}, for $f, g\in L_c^\infty$ we may write
 \begin{align*}
\bla [T_1, [b, T_2]]f, g\bra= - \iint_{\R^{d_2}\times \R^{d_2}} K_2(x_2, y_2) \big\langle  \big[ b(x_2)-b(y_2) , T_1\big]f(y_2), g(x_2) \big\rangle_{\R^{d_1}}.
 \end{align*}
By the sparse domination of commutators -- see \cite{LOR1} -- we have for all $r \in (1,\infty)$ that
\begin{align*}
  \big|\big\langle  \big[ b(x_2)-b(y_2) &, T_1\big]f(y_2), g(x_2) \big\rangle_{\R^{d_1}} \big|\\
  & \lesssim \| b(\cdot, x_2)-b(\cdot, y_2)\|_{\BMO(\R^{d_1})}\int_{\R^{d_1}}M_r^1 f(\cdot, y_2) M_r^1 g(\cdot, x_2).
\end{align*}
Choosing $1<r<\min\{p_1, p_2, p_1', q_2'\}$ and using the above we have that $\big|\bla [T_1, [b, T_2]]f, g\bra\big|$ can be dominated by
\begin{align*}
\|b\|_{\dot C^{0, \beta_2}_{x_2}(\BMO_{x_1})}\iint_{\R^{d_2}\times \R^{d_2}}& |K_2(x_2, y_2)| |x_2-y_2|^{\beta_2}\int_{\R^{d_1}}M_r^1 f(x_1, y_2) M_r^1 g(x_1, x_2) \\
 &\le \|b\|_{\dot C^{0, \beta_2}_{x_2}(\BMO_{x_1})} \|f\|_{L^{p_1}_{x_1}L^{p_2}_{x_2}}\|g\|_{L^{p_1'}_{x_1}L^{q_2'}_{x_2}},
\end{align*} where we have used the boundedness of fractional integrals and mixed norm estimates of $M_r^1$. This completes the proof.
\end{proof}
We move on to the upper bounds that are related to the H\"older space estimates of Proposition \ref{prop:CLrBMOLr}.
\begin{prop}
If $p_1 < q_1$ and $p_2 > q_2$ then
$$
\|[T_1, [b, T_2]]\|_{L^{p_1}_{x_1}L^{p_2}_{x_2}\to L^{q_1}_{x_1}L^{q_2}_{x_2}} \lesssim \|b\|_{\dot C^{0,\beta_1}_{x_1}(\dot L^{r_2}_{x_2})} 
$$
If $p_1 > q_1$ and $p_2 < q_2$ then
$$
\|[T_1, [b, T_2]]\|_{L^{p_1}_{x_1}L^{p_2}_{x_2}\to L^{q_1}_{x_1}L^{q_2}_{x_2}} \lesssim \|b\|_{\dot L^{r_1}_{x_1}(\dot C^{0,\beta_2}_{x_2})}.
$$
\end{prop}
\begin{proof}
Let $p_1 < q_1$ and $p_2 > q_2$. Similarly as above for $f, g\in L_c^\infty$ we may write
$$
\bla [T_1, [b, T_2]]f, g\bra= - \iint_{\R^{d_1}\times \R^{d_1}} K_1(x_1, y_1) \big\langle  \big[ b(x_1)-b(y_1) , T_2\big]f(y_1), g(x_1) \big\rangle_{\R^{d_2}}.
$$
This gives that
\begin{align*}
  &\big|\bla [T_1, [b, T_2]]f, g\bra\big|\\&= \left|\iint_{\R^{d_1}\times \R^{d_1}}|x_1-y_1|^{\beta_1}K_1(x_1, y_1) \Big\langle  \Big[\frac{b(x_1)-b(y_1)}{|x_1-y_1|^{\beta_1}}, T_2\Big]f(y_1), g(x_1)\Big\rangle_{\R^{d_2}}\right|\\
  &\lesssim \iint_{\R^{d_1}\times \R^{d_1}}|x_1-y_1|^{\beta_1-d_1}
   \BNorm{\frac{b(x_1)-b(y_1)}{|x_1-y_1|^{\beta_1}}}{\dot L^{r_2}_{x_2}} 
   \Norm{f(y_1)}{L^{p_2}_{x_2}}\Norm{g(x_1)}{L^{q_2'}_{x_2}} \\
  &\lesssim \|b\|_{\dot C^{0, \beta_1}_{x_1}(\dot L^{r_2}_{x_2})} \|f\|_{L^{p_1}_{x_1}L^{p_2}_{x_2}}\|g\|_{L^{q_1'}_{x_1}L^{q_2'}_{x_2}},
\end{align*}
where we have used the elementary inequality \eqref{eq:propofbT2}
and the boundedness of fractional integrals.

Next, let $p_1 > q_1$ and $p_2 < q_2$.
We simply estimate
\begin{align*}
\| [T_1, [b, T_2]] f\|_{L^{q_1}_{x_1}L^{q_2}_{x_2}} &= \| [T_1, [b-c_2, T_2]] f\|_{L^{q_1}_{x_1}L^{q_2}_{x_2}}\\ &\le \| T_1 [b-c_2, T_2] f\|_{L^{q_1}_{x_1}L^{q_2}_{x_2}}+ \|  [b-c_2, T_2] T_1 f\|_{L^{q_1}_{x_1}L^{q_2}_{x_2}},
\end{align*}
use the mixed norm estimates of $T_1$, and the estimate
\[
\| [b-c_2, T_2]f\|_{L^{q_1}_{x_1}L^{q_2}_{x_2}}\lesssim \left\|\|b-c_2\|_{\dot C^{0, \beta_2}_{x_2}}\|f\|_{L^{p_2}_{x_2}}\right\|_{L^{q_1}_{x_1}}\le \|b-c_2\|_{L^{r_1}_{x_1}(\dot C^{0, \beta_2}_{x_2})}  \|f\|_{L^{p_1}_{x_1}L^{p_2}_{x_2}}.
\]
To end the proof we take the infimum over suitable $c_2(x) = c_2(x_2)$.
\end{proof}

\subsection{$\BMO$ and $\dot L^r$}
Proposition \ref{prop:CLrBMOLr} recorded the oscillatory lower bounds related to the space $\BMO(\dot L^r)$. Here are the related upper bounds.
\begin{prop}\label{prop:usesVecVal}
If $p_1 = q_1$ and $p_2 > q_2$, then
$$
\|[T_1, [b, T_2]]\|_{L^{p_1}_{x_1}L^{p_2}_{x_2}\to L^{p_1}_{x_1}L^{q_2}_{x_2}}\lesssim \|b\|_{\BMO_{x_1}(\dot L^{r_2}_{x_2})}.
$$
If $p_1>q_1$ and $p_2=q_2$, then
$$
\|[T_1, [b, T_2]]\|_{L^{p_1}_{x_1}L^{p_2}_{x_2}\to L^{q_1}_{x_1}L^{p_2}_{x_2}} \lesssim \|b\|_{\dot L^{r_1}_{x_1}(\BMO_{x_2})}.
$$
\end{prop}
\begin{proof}
The first estimate, where $p_1 = q_1$ and $p_2 > q_2$, requires some vector-valued theory of commutators.
We notice that
\begin{equation}\label{eq:modvec}
([T_1, [b, T_2]]f)^{\sharp}\lesssim_{\varepsilon} \|b\|_{\BMO_{x_1}( \dot L^{r_2}_{x_2} )}\big( M_{1+\varepsilon}\|T_1f\|_{L^{p_2}_{x_2}} + M_{1+\varepsilon}\|f\|_{L^{p_2}_{x_2}} \big),
\end{equation}
where $M$ is the maximal function, $M_s g = (M|g|^s)^{1/s}$ and the sharp maximal function is defined using the $L^{q_2}_{x_2}$ norm as follows
\[
g^{\sharp}(x_1) := \sup_{I_1} \fint_{I_1} \|g -\langle g \rangle_{I_1}\|_{L^{q_2}_{x_2}}.
\]
Here the supremum is over the cubes $I_1 \subset \R^{d_1}$ centered at $x_1$.
The proof of this is essentially a vector-valued version of a known pointwise bound for the sharp maximal function of $[b,T_1]$, but we need to keep 
the operator $T_2$ around to exploit the elementary estimate
\begin{equation}\label{eq:propofbT2}
\|[b, T_2]f\|_{L^{q_2}_{x_2}}\lesssim \|b\|_{\dot L^{r_2}_{x_2}}\|f\|_{L^{p_2}_{x_2}}.
\end{equation}
We give the full details of \eqref{eq:modvec} for the convenience of the reader.
Below we repeatedly use the well-known vector-valued boundedness of $T_1$ and \eqref{eq:propofbT2} without explicit mention.

Fix a cube $I_1$ containing the implicit variable. Let $f_1= f 1_{I_1^*}$ and $f_2=f-f_1$, where $I_1^*=5\sqrt{d_1} I_1$. Let
\[
c= \bla T_1([b-a, T_2]f_2)\bra_{I_1, 1}, \qquad a=\langle b\rangle_{I_1^*, 1}.
\] We can now write 
\[
[T_1, [b, T_2]]f= [T_1, [b-a, T_2]]f
= 
-[b-a, T_2]T_1 f+ T_1 ([b-a, T_2]f_1)+  T_1 ([b-a, T_2]f_2).
\]Using this we split
\begin{align*}
\fint_{I_1} \| [T_1, [b, T_2]]f-c \|_{L^{q_2}_{x_2}}&\le \fint_{I_1} \| [b-a, T_2]T_1 f \|_{L^{q_2}_{x_2}}+\fint_{I_1} \| T_1 ([b-a, T_2]f_1) \|_{L^{q_2}_{x_2}}\\
&+ \fint_{I_1} \| T_1 ([b-a, T_2]f_2)-c\|_{L^{q_2}_{x_2}}\\
&=: L_1+L_2+L_3.
\end{align*}

We have
\[
L_1\lesssim\Big(\fint_{I_1^*} \|b-a\|_{ \dot L^{r_2}_{x_2}}^{(1+\varepsilon)'}\Big)^{\frac 1{(1+\eps)'}}\Big(\fint_{I_1} \|T_1f\|_{  L^{p_2}_{x_2}}^{1+\varepsilon}\Big)^{\frac 1{1+\eps}}\lesssim \|b\|_{\BMO_{x_1}(  \dot L^{r_2}_{x_2})}  M_{1+\varepsilon}(\|T_1f\|_{L^{p_2}_{x_2}}),
\]where in the last inequality we have used John-Nirenberg inequality of $\BMO_{x_1}(  \dot L^{r_2}_{x_2})$, which has the same proof
as the usual vector-valued John-Nirenberg inequality (the dot in $\dot L^{r_2}_{x_2}$ does not matter).

We move to $L_2$. This time we have
\begin{align*}
L_2& \lesssim |I_1|^{-\frac 1{1+\frac \eps 2}}\| [b-a, T_2]f_1\|_{L^{1+\frac \eps 2}_{x_1}(L^{q_2}_{x_2})}\\
&\lesssim \Big(\fint_{I_1^*} \|b-a\|_{ \dot L^{r_2}_{x_2}}^{1+\frac \eps 2} \|f\|_{L^{p_2}_{x_2}}^{1+\frac \eps 2}\Big)^{\frac 1{1+\frac \eps 2}}\lesssim \|b\|_{\BMO_{x_1}(\dot L^{r_2}_{x_2})}  M_{1+\varepsilon}\|f\|_{L^{p_2}_{x_2}}.
\end{align*}
We used John-Nirenberg again in the last step.

It remains to estimate $L_3$. Since $I_1 \cap \supp f_2=\emptyset$, using the kernel representation of $T_1$ and the regularity of the kernel we have
\begin{align*}
L_3&\lesssim \sum_{j=0}^\infty 2^{-j\alpha_1}\fint_{2^{j+1}I_1^*} \|[b-a,T_2]f\|_{L^{q_2}_{x_2}}\lesssim \sum_{j=0}^\infty 2^{-j\alpha_1}\fint_{2^{j+1}I_1^*} \|b-a\|_{ \dot L^{r_2}_{x_2}} \|f\|_{L^{p_2}_{x_2}}\\&
\le \sum_{j=0}^\infty 2^{-j\alpha_1}\fint_{2^{j+1}I_1^*} \|b-\langle b\rangle_{2^{j+1}I_1^*, 1}\|_{ \dot L^{r_2}_{x_2}} \|f\|_{L^{p_2}_{x_2}} \\ &+ \sum_{j=0}^\infty 2^{-j\alpha_1}\fint_{2^{j+1}I_1^*} \|\langle b\rangle_{I_1^*,1}-\langle b\rangle_{2^{j+1}I_1^*, 1}\|_{\dot L^{r_2}_{L^{p_2}_{x_2}}} \|f\|_{L^{p_2}_{x_2}}.
\end{align*}
Using John-Nirenberg again, the first sum can be bounded by $\|b\|_{\BMO_{x_1}( \dot L^{r_2}_{x_2})}  M_{1+\varepsilon}\|f\|_{L^{p_2}_{x_2}}$. Moreover, since 
\[
\|\langle b\rangle_{I_1^*,1}-\langle b\rangle_{2^{j+1}I_1^*, 1}\|_{\dot L^{r_2}_{x_2}}\le \sum_{i=0}^{j} \|\langle b\rangle_{2^{i+1}I_1^*,1}-\langle b\rangle_{2^{i}I_1^*,1}\|_{ \dot L^{r_2}_{x_2}}\lesssim j \|b\|_{\BMO_{x_1}(\dot L^{r_2}_{x_2})},
\]the desired estimate follows for the second term as well. The inequality \eqref{eq:modvec} follows.

Next, if $\varepsilon$ is so small that $p_1/(1+\varepsilon) > 1$, we have by the standard Fefferman-Stein inequality that
\begin{align*}
\|[T_1, [b, T_2]]f\|_{L^{p_1}_{x_1}L^{q_2}_{x_2}}^{p_1}
&= \int_{\R^{d_1}} \|[T_1, [b, T_2]]f\|_{L^{q_2}_{x_2}}^{p_1} \\
&\lesssim \int_{\R^{d_1}} \big[ ([T_1, [b, T_2]]f)^{\sharp} \big]^{p_1} \\
&\lesssim \|b\|_{\BMO_{x_1}( \dot L^{r_2}_{x_2} )}^{p_1} \big( \|M_{1+\varepsilon}\|T_1f\|_{L^{p_2}_{x_2}}\|_{L^{p_1}_{x_1}}^{p_1}+\|M_{1+\varepsilon}\|f\|_{L^{p_2}_{x_2}} \|_{L^{p_1}_{x_1}}^{p_1} \big)\\
&\lesssim \|b\|_{\BMO_{x_1}( \dot L^{r_2}_{x_2} )}^{p_1} \|f\|_{L^{p_1}_{x_1} L^{p_2}_{x_2}}^{p_1}.
\end{align*}
We are done with the case $p_1 = q_1$ and $p_2 > q_2$. See Section \ref{sec:vecval} for more about vector-valued estimates of commutators.

The case $p_1 > q_1$ and $p_2 = q_2$ is the following elementary deduction. Estimate
\begin{align*}
\| [T_1, [b, T_2]] f\|_{L^{q_1}_{x_1}L^{p_2}_{x_2}} &= \| [T_1, [b-c_2, T_2]] f\|_{L^{q_1}_{x_1}L^{p_2}_{x_2}} \\
&\le \| T_1 [b-c_2, T_2] f\|_{L^{q_1}_{x_1}L^{p_2}_{x_2}}+ \|  [b-c_2, T_2] T_1 f\|_{L^{q_1}_{x_1}L^{p_2}_{x_2}},
\end{align*}
use the mixed norm estimates of $T_1$ and 
\[
\| [b-c_2, T_2]f\|_{L^{q_1}_{x_1}L^{p_2}_{x_2}}\lesssim \left\|\|b-c_2\|_{\BMO_{x_2}}\|f\|_{L^{p_2}_{x_2}}\right\|_{L^{q_1}_{x_1}}\le \|b-c_2\|_{L^{r_1}_{x_1}(\BMO_{x_2})} \|f\|_{L^{p_1}_{x_1}L^{p_2}_{x_2}}.
\]
To end the proof we take the infimum over suitable $c_2(x) = c_2(x_2)$.
\end{proof}

The oscillatory lower bound related to $\dot L^{r_1} \dot L^{r_2}$ was recorded in Proposition \ref{prop:dotdot}. The upper bound is completely elementary and does not utilize any
cancellation of the bi-commutator.
\begin{prop}
If $p_1>q_1$ and $p_2>q_2$, then
$$
\|[T_1, [b, T_2]]\|_{L^{p_1}_{x_1}L^{p_2}_{x_2}\to L^{q_1}_{x_1}L^{q_2}_{x_2}} \lesssim \|b\|_{\dot L^{r_1}_{x_1}(\dot L^{r_2}_{x_2})}.
$$
\end{prop}
\begin{proof}
Simply use H\"older's inequality and the mixed-norm estimates of $T_1$ and $T_2$.
\end{proof}

\section{Vector-valued estimates for commutators}\label{sec:vecval}
In the proof of Proposition \ref{prop:usesVecVal} we used some vector-valued theory of commutators. Although the proof of Proposition \ref{prop:usesVecVal} was completely
self-contained, in this section we more systematically
prove, and partly recall, such vector-valued estimates. This topic is of independent interest.

\subsection{Vector-valued estimates for $[b,T]$}

In the one-parameter setting $[b,T]$, where $T$ is a CZO in $\R^d$, the vector-valued estimates are known and quite easy, but perhaps not so well-known in the community.
They are simple as there is a certain pointwise estimate for the sharp maximal function of $[b,T]$, which has a straightforward vector-valued extension. A reference
is \cite{ST}. This idea was already used in the proof of Proposition \ref{prop:usesVecVal}. We now recall how this can be done in a general situation.

First, a few definitions. For an extensive treatment of Banach space theory see the books \cite{HNVW1, HNVW2} by Hyt\"onen, van Neerven, Veraar and Weis.

Given a Banach space $X$ with norm $|\cdot|_X$ we
let $L^p(\R^d;X)$, $p \in (0, \infty]$, be the Bochner space of $X$-valued functions with
$\int_{\R^d} |f(x)|_X^p \ud x < \infty$. 
For $f:\R^d \to X$, we also define
\[
f^\sharp(x) := \sup_{Q=Q(x,r), r>0} \frac 1{|Q|}\int_Q |f-\langle f\rangle_Q |_X,
\]
where the supremum is over the cubes $Q(x,r)$ centred at $x$.
We have
\[
f^\sharp(x)\sim \sup_{Q=Q(x,r), r>0} \inf_{e\in X} \frac 1{|Q|}\int_Q |f-e |_X.
\]

The $\UMD$ property is a necessary and sufficient condition for the boundedness of various one-parameter singular integrals on $L^p(\R^d;X) = L^p(X)$, see \cite[Sec. 5.2.c and the Notes to Sec. 5.2]{HNVW1}. 
We essentially only need to blackbox this fact, but below is the definition.
\begin{defn}
A Banach space $X$ is said to be a $\UMD$ \emph{space}, where $\UMD$ stands for unconditional martingale differences, if for all $p \in (1,\infty)$, all
$X$-valued $L^p$-martingale difference sequences $(d_j)_{j=1}^k$ and all choices of signs $\varepsilon_j \in \{-1,1\}$
we have
\begin{equation}\label{eq:UMDDef}
\Big\| \sum_{j=1}^k \varepsilon_j d_j \Big \|_{L^p(X)}
\lesssim \Big\| \sum_{j=1}^k d_j \Big \|_{L^p(X)}.
\end{equation}
The $L^p(X)$-norm is with respect to the measure space where the martingale differences are defined.
\end{defn}
A standard property of $\UMD$ spaces is that if \eqref{eq:UMDDef} holds for one $p_0\in (1,\infty)$ it
holds for all $p \in (1, \infty)$ \cite[Theorem 4.2.7]{HNVW1}. Moreover, if $X$ is $\UMD$ then so is the dual space $X^*$ \cite[Prop. 4.2.17]{HNVW1}.

Suppose now that $X_1$ is a Banach space and $ X_2, X_3$ are $\UMD$ spaces such that there exists a bilinear map $X_1\times X_2\to X_3$, which we denote by
\[
(x_1,x_2) \to x_1x_2,
\]
so that $$|x_1x_2|_{X_3}\le |x_1|_{X_1}|x_2|_{X_2}.$$ Let $f$ take values in $X_2$ and $b$ take values in $X_1$. For a CZO $T$ on $\R^d$ a standard estimate
as in the proof of Proposition \ref{prop:usesVecVal}
gives that for all $\varepsilon > 0$ we have
\begin{equation}\label{eq:sharp}
\big([b, T]f\big)^\sharp\lesssim_{\varepsilon} \|b\|_{\BMO(X_1)}\big( M_{1+\varepsilon}(|Tf|_{X_2})+ M_{1+\varepsilon}|f|_{X_2} \big),
\end{equation}
where $M$ is the maximal function and $M_s g = (M|g|^s)^{1/s}$. The sharp maximal function is defined using the norm of $X_3$.
To get this estimate, already the boundedness $T \colon L^q(X_3) \to L^q(X_3)$, $q \in (1,\infty)$, is required
-- this is true as $X_3$ is $\UMD$.
The vector-valued estimate
\begin{equation}\label{eq:vecval1par}
\|[b, T]f\|_{L^p(X_3)} \lesssim \|b\|_{\BMO(X_1)} \|f\|_{L^p(X_2)}, \qquad p \in (1,\infty),
\end{equation}
follows from \eqref{eq:sharp} using the Fefferman-Stein inequality as in the proof of Proposition \ref{prop:usesVecVal}.
Here the $\UMD$ property of $X_2$ is needed via the boundedness $T \colon L^p(X_2) \to L^p(X_2)$.

Notice that the proof of Proposition \ref{prop:usesVecVal} was almost just this vector-valued estimate -- we just needed to work
a bit more to get the norm $\|b\|_{\BMO_{x_1}(\dot L^{r_2}_{x_2})}$ instead of $\|b\|_{\BMO_{x_1}(L^{r_2}_{x_2})}$.

Alternatively, it is also possible to directly prove a vector-valued version of the sparse domination of commutators \cite{LOR1}:
\begin{equation*}\label{eq:sparsecom}
\begin{split}
|\langle [b,T]f, g \rangle| &\lesssim \sum_{Q \in \calS} \bla |b-\langle b \rangle_Q|_{X_1}  |f|_{X_2}  \bra_{Q} \langle |g|_{X_3^*} \rangle_Q |Q| \\
&+ \sum_{Q \in \calS}\langle |f|_{X_2} \rangle_Q  \bla |b-\langle b \rangle_Q|_{X_1}  |g|_{X_3^*} \bra_{Q}  |Q|.
\end{split}
\end{equation*}

\subsection{Vector-valued estimates for bi-commutators}
We will prove the bi-parameter analogue of the vector-valued estimate \eqref{eq:vecval1par}. There is no equally cheap way for this as in the above
one-parameter case. However, we want to show that it can still be done with different methods.

We need more definitions
and tools.

\subsubsection*{Dyadic notation}\label{sec:DyadicNotation}
Let $\calD$ be some fixed dyadic lattice in $\R^d$.
For a fixed $Q \in \calD$ and $f \in L^1_{\loc}(X)$ we define as follows.
\begin{itemize}
\item If $k \in \Z$, $k \ge 0$, then $Q^{(k)}$ denotes the unique cube $S \in \calD$ for which $Q \subset S$ and
$\ell(Q) = 2^{-k}\ell(S)$.
\item The dyadic children of $Q$ are denoted by $\ch (Q) = \{Q' \in \calD\colon (Q')^{(1)} = Q\}$.
\item The average operator is $E_Q f=\langle f \rangle_Q 1_Q$.
\item The martingale difference $\Delta_Q f$ is defined by $\Delta_Q f= \sum_{Q' \in \ch (Q)} E_{Q'} f - E_{Q} f$.
\item For $k \in \Z$, $k \ge 0$, we define the martingale difference block
$$
\Delta_{Q,k} f=\sum_{\substack{S \in \calD \\ S^{(k)}=Q}} \Delta_{S} f.
$$
\end{itemize}

Haar functions are used to decompose martingale differences $\Delta_Q f$ in terms of rank-one operators.
For an interval $I \subset \R$, we denote by $I_{l}$ and $I_{r}$ the left and right
halves of the interval $I$, respectively. We define $h_{I}^0 = |I|^{-1/2}1_{I}$ and $h_{I}^1 = |I|^{-1/2}(1_{I_{l}} - 1_{I_{r}})$.
Let now $Q = I_1 \times \cdots \times I_d \in \calD$, and define the Haar function $h_Q^{\eta}$, $\eta = (\eta_1, \ldots, \eta_d) \in \{0,1\}^d$, via
\begin{displaymath}
h_Q^{\eta} = h_{I_1}^{\eta_1} \otimes \cdots \otimes h_{I_d}^{\eta_d}.
\end{displaymath}
If $\eta \ne 0$, the Haar function is cancellative: $\int h_Q^{\eta} = 0$. We may write $\Delta_Q f = \sum_{\eta \ne 0} \langle f, h_{Q}^{\eta}\rangle h_{Q}^{\eta}$,
where $\langle f, h_Q^{\eta} \rangle = \int f h_Q^{\eta}$.
We exploit notation by suppressing the presence of $\eta$, and simply write $h_Q$ for some $h_Q^{\eta}$, $\eta \ne 0$.
Similarly, we write $\Delta_Q f = \langle f, h_Q \rangle h_Q$.

When performing bi-parameter analysis and thinking of $\R^d$ as a product space $\R^{d_1} \times \R^{d_2}$, we denote
a dyadic grid in $\R^{d_i}$ by $\calD^i$ and denote the related dyadic rectangles in $\R^d$ by $\calD = \calD^1 \times \calD^2$.
If $R = I_1 \times I_2 \in \calD$ we set $h_R = h_{I_1} \otimes h_{I_2}$. We also define $\Delta_{I_1}^1 f(x) := \Delta_{I_1} (f(\cdot, x_2))(x_1)$ and
define $\Delta_{I_2}^2f$ analogously. Then we set $\Delta_R f = \Delta_{I_1}^1 \Delta_{I_2}^2 f = \Delta_{I_2}^2 \Delta_{I_1}^1 f$.
We record that $\Delta^1_{I_1} f = h_{I_1} \otimes \langle f , h_{I_1} \rangle_1$, $\Delta^2_{I_2} f = \langle f, h_{I_2} \rangle_2 \otimes h_{I_2}$ and
$ \Delta_R f = \langle f, h_R\rangle h_R$.
Bi-parameter martingale blocks may also be defined in the natural way
$$
\Delta_R^{j_1, j_2} f  =  \sum_{J_1 \colon J_1^{(j_1)} = I_1} \sum_{J_2 \colon J_2^{(j_2)} = I_2} \Delta_{J_1 \times J_2} f.
$$

\subsubsection*{More about Banach spaces}

In addition to $\UMD$ we will need the ``property $(\alpha)$''. Bi-parameter extensions require this -- see \cite[Sec. 8.3.e]{HNVW2}.
Suppose that for all $N$, all scalars $a_{i,j}$ and all
$e_{i,j} \in X$, $1 \le i, j \le N$, there holds
$$
\Big(\E \E' \Big| \sum_{1 \le i, j \le N} \varepsilon_i \varepsilon_j' a_{i,j} e_{i,j} \Big|_X^2\Big)^{1/2}
\lesssim \max_{i,j} |a_{i,j}| \Big(\E \E' \Big| \sum_{1 \le i, j \le N} \varepsilon_i \varepsilon_j' e_{i,j} \Big|_X^2 \Big)^{1/2}.
$$
If this holds, the Banach space $X$ is said to satisfy the property $(\alpha)$ of Pisier.
Here $\{\varepsilon_k\}_k$ is a \emph{collection of independent random signs} -- that is, the following holds.
We have $\varepsilon_k \colon \calM \to \{-1,1\}$, where $(\calM, \rho)$ is a probability space, the collection
$\{\varepsilon_k\}_k$ is independent and
$\rho(\{\varepsilon_k=1\})=\rho(\{\varepsilon_k=-1\})=1/2$.

Instead of dealing with $\UMD$ spaces with property $(\alpha)$ we will be
working with the formally stronger (although for concrete examples essentially the same) assumption that $X$ is a $\UMD$ function lattice. For proofs and background on function lattices see
\cite{ALV, HMV, LT, Lo}.
A normed space $X$ is a Banach function lattice if the following four conditions hold.
Let $(\Omega, \mathcal{A}, \mu)$ be a $\sigma$-finite measure space.
\begin{enumerate}
\item Every $e \in X$ is a measurable function $e \colon \Omega \to \R$ (an equivalence class).
\item If $e \colon \Omega \to \R$ is measurable, $u \in X$ and $|e(\omega)| \le |u(\omega)|$ for $\mu$-a.e. $\omega \in \Omega$, then
$e \in X$ and $|e|_X \le |u|_X$.
\item There is an element $e \in X$ so that $e > 0$ (i.e. $e(\omega) > 0$ for $\mu$-a.e. $\omega \in \Omega$).
\item If $e_i ,e$ are non-negative, $e_i \in X$, $e_i \le e_{i+1}$, $e_i(\omega) \to e(\omega)$ for $\mu$-a.e. $\omega \in \Omega$ and $\sup_i |e_i|_X < \infty$,
we have $e \in X$ and $|e_i|_X \to |e|_X = \sup_i |e_i|_X$.
\end{enumerate}
Such a space $X$ is automatically a Banach space.
For a measurable $u \colon \Omega \to \R$ with $eu \in L^1(\mu)$ for all $e \in X$, we define
$$
\Lambda_u \colon X \to \R, e_u(e) = \int_{\Omega} e(\omega)u(\omega) \ud \mu(\omega).
$$
In this case $\Lambda_u \in X^*$ (the dual space of $X$). We set $X' \subset X^*$ to consist of those elements of $X^*$ that have the form $\Lambda_u$ for some $u$ like above and make the obvious identification.
The Banach function lattice $X'$ is called the K\"othe dual of $X$.

We mostly work with $\UMD$ Banach function lattices, where $X' = X^*$ is automatic.
Indeed, $\UMD$ spaces are always reflexive,
and in reflexive Banach lattices $X' = X^*$ holds. When $X' = X^*$ is not automatic, we prefer to tacitly assume it for clarity.
A $\UMD$ function lattice $X$ also automatically satisfies the property $(\alpha)$ of Pisier -- this explains why this assumption is not explicitly needed.


Next, we explain some maximal function estimates by Bourgain \cite{Bo} and Rubio de Francia \cite{Ru} that are valid in $\UMD$ function lattices (see also \cite{GMT}).
We also explain some standard square function estimates, see e.g. the recent paper \cite{HMV}.

Let $X$ be a $\UMD$ function lattice, $\R^d = \R^{d_1} \times \R^{d_2}$ and $\calD = \calD^{1} \times \calD^{2}$, where $\calD^{i}$ is a dyadic grid in $\R^{d_i}$.
Given $f \colon \R^{d} \to X$ and $g \colon \R^{d_1} \to X$ we define the dyadic (lattice) maximal functions
$$
M_{\calD^1}g(x_1, \omega) := \sup_{I \in \calD^1} \frac{1_I(x_1)}{|I|}\int_I |g(x_1, \omega)|\ud x_1 \textup{ and }
M_{\calD} f(x, \omega) := \sup_{R \in \calD}  \frac{1_R(x)}{|R|}\int_R |f(x, \omega)|\ud x.
$$
We also set $M^1_{\calD^1} f(x_1, x_2) =  M_{\calD^1}(f(\cdot, x_2))(x_1)$. The operator $M^2_{\calD^2}$ is defined similarly.

Similarly, define the square functions
$$
S_{\calD} f = \Big( \sum_{R \in \calD}  |\Delta_{R} f|^2 \Big)^{1/2}, \,\, S_{\calD^1}^1 f =  \Big( \sum_{I \in \calD^n}  |\Delta_I^1 f|^2 \Big)^{1/2}\, \textup{ and } \,
S_{\calD^2}^2 f =  \Big( \sum_{J \in \calD^m} |\Delta_J^2 f|^2 \Big)^{1/2}.
$$
Notice that $|x| \in X$ and $|x|^2 \in X$ are defined in the natural pointwise sense, as $X$ is a function lattice.
Define also
$$
S_{\calD^1, M_{\calD^2}} f = \Big( \sum_{I \in \calD^1} \frac{1_I}{|I|} \otimes [M_{\calD^2} \langle f, h_I \rangle_1]^2 \Big)^{1/2}
$$
and define $S_{\calD^2, M_{\calD^1}} f$ analogously.
We are ready to state the following standard estimates that are at least recorded in \cite{HMV}.
\begin{lem}\label{lem:standardEst}
Let $\calD = \calD^{1} \times \calD^{2}$, where $\calD^{i}$ is a dyadic grid in $\R^{d_i}$, and let $X$ be a $\UMD$ function lattice.
Then we have for all $p_1, p_2 \in (1, \infty)$ that
$$
\| f \|_{L^{p_1}_{x_1} L^{p_2}_{x_2}(X)}
 \sim \| S_{\calD} f\|_{L^{p_1}_{x_1} L^{p_2}_{x_2}(X)}
\sim  \| S_{\calD^1}^1 f  \|_{L^{p_1}_{x_1} L^{p_2}_{x_2}(X)}
\sim \| S_{\calD^2}^2 f  \|_{L^{p_1}_{x_1} L^{p_2}_{x_2}(X)}.
$$
Moreover, for $s \in (1,\infty)$ we have the Fefferman--Stein inequality
$$
\Big\| \Big( \sum_j |M f_j |^s \Big)^{1/s} \Big\|_{L^{p_1}_{x_1} L^{p_2}_{x_2}(X)} \lesssim \Big\| \Big( \sum_{j} | f_j |^s \Big)^{1/s} \Big\|_{L^{p_1}_{x_1} L^{p_2}_{x_2}(X)}.
$$
Here $M$ can e.g. be $M_{\calD^1}^1$ or $M_{\calD}$. Finally, we have
$$
\| S_{\calD^1, M_{\calD^2}} f\|_{L^{p_1}_{x_1} L^{p_2}_{x_2}(X)} + \| S_{\calD^2, M_{\calD^1}}f\|_{L^{p_1}_{x_1} L^{p_2}_{x_2}(X)}\lesssim \|f\|_{L^{p_1}_{x_1} L^{p_2}_{x_2}(X)}.
$$
\end{lem}

\subsubsection*{Product $\BMO$ and vector-valued $H^1-\BMO$ estimates}

Let again $X$ be a Banach function lattice, $\R^d = \R^{d_1} \times \R^{d_2}$ and $\calD = \calD^{1} \times \calD^{2}$, where $\calD^{i}$ is a dyadic grid in $\R^{d_i}$.
A locally integrable function $b \colon \R^d \to X$
belongs to the dyadic $X$-valued product $\BMO$ space $\BMO_{\calD, \textup{prod}}(p, X)$, $0 < p < \infty$, if
$$
\sup_{\Omega}  \frac{1}{|\Omega|^{1/p}}  \Big\| \Big( \sum_{ \substack{ R \in \calD \\ R \subset \Omega }} |\langle b, h_R \rangle|^2 \frac{1_R}{|R|}
\Big)^{1/2}\Big\|_{L^p(X)}   < \infty,
$$
where the supremum is taken over all open subsets $\Omega \subset \R^{d}$ of finite measure. We define $\BMO_{\calD, \textup{prod}}(X) := \BMO_{\calD, \textup{prod}}(2,X)$. With
this definition we will not require a Banach-valued bi-parameter John--Nirenberg inequality for which we do not have a convenient reference.
In the scalar-valued case it is well-known that all these norms are equivalent. The product $\BMO$ is the supremum over all dyadic grids
of the dyadic product $\BMO$ norms.

\begin{rem}
The following explanation gives one way to arrive at the definition of product $\BMO$ in our generality.
If $X$ is $\UMD$ and $p \in (1,\infty)$ then
\begin{align*}
\Big( \int_{J_i} |b - \langle b \rangle_{J_i}|_X^p \Big)^{1/p} = \Big\|  \sum_{ \substack{ I_i \in \calD^i \\ I_i \subset J_i }} \langle b, h_{I_i} \rangle h_{I_i} \Big\|_{L^p(X)} 
\sim \E \Big\| \sum_{ \substack{ I_i \in \calD^i \\ I_i \subset J_i }} \eps_{I_i} \langle b, h_{I_i} \rangle \frac{1_{I_i}}{|I_i|^{1/2}} \Big\|_{L^p(X)}.
\end{align*}
If $X$ is a $\UMD$ function lattice we further have
$$
\E \Big\| \sum_{ \substack{ I_i \in \calD^i \\ I_i \subset J_i }} \eps_{I_i} \langle b, h_{I_i} \rangle \frac{1_{I_i}}{|I_i|^{1/2}} \Big\|_{L^p(X)} \sim
\Big\| \Big( \sum_{ \substack{ I_i \in \calD^i \\ I_i \subset J_i }} |\langle b, h_{I_i} \rangle|^2 \frac{1_{I_i}}{|I_i|} \Big)^{1/2} \Big\|_{L^p(X)}.
$$
The last similarity can be seen as follows. The Kahane--Khintchine inequality (see \cite[Theorem 6.2.4]{HNVW1}) says if $X$ is a Banach space, we have
for all $x_1,\dots, x_M \in X$ and $p,q \in (0,\infty)$ that
\begin{equation*}\label{eq:KK}
\Big(\E \Big | \sum_{m=1}^M \varepsilon_m x_m \Big |_X^p \Big)^{1/p}
\sim \Big(\E \Big | \sum_{m=1}^M \varepsilon_m x_m \Big |_X^q \Big)^{1/q}.
\end{equation*}
Applying this twice we see that
$$
\E \Big\| \sum_{ \substack{ I_i \in \calD^i \\ I_i \subset J_i }} \eps_{I_i} \langle b, h_{I_i} \rangle \frac{1_{I_i}}{|I_i|^{1/2}} \Big\|_{L^p(X)}
\sim \Big\| \Big( \E \Big| \sum_{ \substack{ I_i \in \calD^i \\ I_i \subset J_i }} \eps_{I_i} \langle b, h_{I_i} \rangle \frac{1_{I_i}}{|I_i|^{1/2}} \Big|_X^2\Big)^{1/2} \Big\|_{L^p}.
$$
It remains to use the equivalence
\begin{equation*}\label{eq:bd}
\Big( \E \Big| \sum_j \eps_j e_j \Big|_X^2\Big)^{1/2} \sim \Big|\Big(\sum_j |e_j|^2 \Big)^{1/2}\Big|_X
\end{equation*}
for all $e_j \in X$. This theorem of Khintchine--Maurey holds in all Banach lattices with finite cotype -- in particular, in $\UMD$ function lattices. For a proof see
Theorem 7.2.13 in \cite{HNVW2}.

This explains that in one-parameter the dyadic $\BMO(X)$ norm (if $X$ is $\UMD$) also has the equivalent form
$$
\sup_{J_i \in \calD^i} \Big\| \Big( \sum_{ \substack{ I_i \in \calD^i \\ I_i \subset J_i }} |\langle b, h_{I_i} \rangle|^2 \frac{1_{I_i}}{|I_i|} \Big)^{1/2} \Big\|_{L^p(X)}.
$$
The product $\BMO$ is simply the bi-parameter analogue of this form, where the supremum is taken over all open sets.
The rectangular $\BMO$ is the analogue of this, where the supremum is over dyadic rectangles. This is because
$$
\sum_{ \substack{ P \in \calD \\ P \subset R = I_1 \times I_2 }} \langle b, h_P \rangle h_P = (b-\langle b\rangle_{I_1,1} -\langle b\rangle_{I_2,2} +\langle b\rangle_R)1_R.
$$
\end{rem}

Let $b \colon \R^d \to X$ and $f \colon \R^d \to X'$, where $X$ is a Banach function lattice and $X'$ is the K\"othe dual of $X$.
Denote $\langle b, f \rangle = \int_{\R^d} \{ b(x), f(x)\}_X\ud x$, where
$\{\cdot, \, \cdot\}_X$ denotes the duality pairing between $X$ and $X'$, in other words, we have $\{e, \, e'\}_X = \int_{\Omega} e(\omega)e'(\omega)\ud \mu(\omega)$ for
$e \in X$ and $e' \in X'$.

The following $H^1$-$\BMO$ duality in the lattice case is convenient. The proof is similar to the scalar-valued case.
\begin{lem}\label{lem:H1BMO}
Let $X$ be a Banach function lattice. Then we have
$$
|\langle b, f\rangle| \lesssim \|b\|_{\BMO_{\calD, \textup{prod}}(X)} \| S_{\calD} f \|_{L^1(X')}.
$$
\end{lem}
\begin{proof}
We may assume $\| S_{\calD} f \|_{L^1(X')} < \infty$.
Estimate
$$
|\langle b, f\rangle| \le  \sum_{R \in \calD}\big\{ |\langle b, h_R\rangle|, |\langle f, h_R\rangle| \big\}_X.
$$
Given $k \in \Z$ we define
$$
U_k = \big\{x \colon \big| S_{\calD} f(x)  \big|_{X'} > 2^{-k} \big\},
$$
and $\widehat\calR_k = \{R \in \calD\colon |R \cap U_k| > |R|/2\}$. If $R \in \widehat\calR_k$, then
$R \subset \tilde U_k := \{x \colon M_{\calD}1_{U_k} > 1/2\}$, and, of course, $|\tilde U_k| \lesssim |U_k|$.

For all $R_0 \in \calD$
and $x \in R_0$ we have
$$
\big| S_{\calD} f(x)  \big|_{X'} \ge   \frac{|\langle f, h_{R_0}\rangle|_{X'}}{|R_0|^{1/2}}.
$$
If $|\langle f, h_{R_0}\rangle|_{X'} > 0$, then this implies $R_0 \in \widehat\calR_k$ for all large enough $k$. As $|U_k| \to 0$ when $k \to -\infty$ we also
have that $R_0 \not \in \widehat\calR_k$ for all small enough $k$. Let $\calR_k = \widehat\calR_k \setminus \widehat\calR_{k-1}$, $k \in \Z$, and notice that by above
\begin{align*}
&\sum_{R \in \calD}\big\{ |\langle b, h_R\rangle|, |\langle f, h_R\rangle| \big\}_X  \\
&\le 2\sum_{k \in \Z} \sum_{R \in \calR_k}
\Big\langle |\langle b, h_R\rangle| \frac{1_R}{|R|^{1/2}}, |\langle f, h_R\rangle| \frac{1_R}{|R|^{1/2}} 1_{\tilde U_k}1_{U_{k-1}^c} \Big\rangle \\
&\lesssim \sum_{k \in \Z} \Big\| \Big( \sum_{ \substack{ R \in \calD \\ R \subset \tilde U_k }} |\langle b, h_R \rangle|^2 \frac{1_R}{|R|}
\Big)^{1/2}\Big\|_{L^2(X)} \Big\| \Big( \sum_{R \in \calD} |\langle f, h_R\rangle|^2 \frac{1_R}{|R|} \Big)^{1/2} 1_{\tilde U_k}1_{U_{k-1}^c} \Big\|_{L^2(X')} \\
&\leq \sum_{k \in \Z}\|b\|_{\BMO_{\calD, \textup{prod}}(X)} \abs{\tilde U_k}^{1/2}\times  \Norm{ 2^{-(k-1)}1_{\tilde U_k}}{L^2} \\
&\lesssim \sum_{k \in \Z}\|b\|_{\BMO_{\calD, \textup{prod}}(X)}  2^{-k} \abs{U_k}
\sim  \|b\|_{\BMO_{\calD, \textup{prod}}(X)} \| S_{\calD} f \|_{L^1(X')}.\qedhere
\end{align*}
\end{proof}

\subsubsection*{Paraproducts}
When we represent a singular integral using dyadic model operators \cite{Hy, Hy2}, one type of model operator that appears is a paraproduct. 
However, this will not be the only source of paraproducts in our arguments. In fact, the main source comes from the desire
to expand products of functions $bf$ -- that appear naturally in commutators -- using paraproducts.
In this section we explain how this is done.

For $i \in \{1,2\}$ we define the one-parameter paraproducts
$$
A^i_1(b,f) = \sum_{I_i \in \calD^{i}} \Delta_{I_i}^i b \Delta_{I_i}^i f, \, \,
A^i_2(b,f) = \sum_{I_i \in \calD^{i}} \Delta_{I_i}^i b E_{I_i}^i f \,\, \textup{ and } \,\, A^i_3(b, f) = \sum_{I_i \in \calD^{i}} E_{I_i}^i b \Delta_{I_i}^i f.
$$
By writing $b = \sum_{I_i} \Delta_{I_i}^i b$ and $f = \sum_{J_i} \Delta_{J_i}^i f$, and collapsing sums such as
$\sum_{J_i \colon I_i \subsetneq J_i} \Delta_{J_i}^i f = E_{I_i}^i f$,
we formally have
$$
bf = \sum_{I_i} \Delta_{I_i}^i b \Delta_{I_i}^i f 
 + \sum_{I_i \subsetneq J_i } \Delta_{I_i}^i b \Delta_{J_i}^i f + \sum_{J_i \subsetneq I_i } \Delta_{I_i}^i b \Delta_{J_i}^i f
=
\sum_{j=1}^3 A^i_j(b,f).
$$
Sometimes, when we pair $bf$ with a tensor product of Haar functions, only one of the appearing Haar functions is cancellative.
In this type of a situation we expand using these one-parameter
paraproducts, like above, in the parameter $i \in \{1,2\}$ where the cancellative Haar function appears. 

A feature of such decompositions is that the paraproduct $A^i_3(b, f)$ appears (e.g. such paraproducts do not appear in the dyadic
representation of singular integrals). Indeed, there is an average on $b$, and so no $\BMO$
type philosophy on that parameter can characterize its boundedness. However, these will appear in commutator decompositions and
the cancellation present in the commutator is key to handling such terms. In fact, none of the one-parameter paraproducts
are bounded by themselves with the assumption $b \in \BMO_{\operatorname{prod}}$.

We continue to define bi-parameter paraproducts, which can be obtained by chaining two one-parameter paraproducts.
For $j_1, j_2 \in \{1, 2, 3\}$ define formally
$$
A_{j_1, j_2}(b,f) = A^1_{j_1}A^2_{j_2}(b,f)
$$
so that e.g.
$$
A_{1, 2}(b,f) = \sum_{I_2 \in \calD^{2}} A^1_1(\Delta_{I_2}^2 b, E_{I_2}^2 f) = \sum_{R = I_1 \times I_2 \in \calD} \Delta_R b \Delta_{I_1}^1 E_{I_2}^2 f.
$$
We can again expand the product $bf$ as follows
$$
bf = \sum_{j_1, j_2 \in \{1,2,3\}} A_{j_1, j_2}(b,f).
$$
Now, those bi-parameter paraproducts where $j_1, j_2 \in \{1,2\}$, will be automatically bounded if $b$ is in product $\BMO$.
The rest will have to be handled by exploiting the cancellation of the commutator.

We now formulate the vector-valued setting, where we can prove the boundedness of our paraproducts.
We say that $(X_1, X_2, X_3)$ is a compatible triple of function lattices, if
each $X_i$ is a Banach lattice defined using the same
underlying measure space $(\Omega, \mu)$, $X_2$ and $X_3$ are $\UMD$, and we have
$$
|x_1x_2|_{X_3} \le |x_1|_{X_1} |x_2|_{X_2}.
$$
\begin{prop}\label{prop:ParaLattice}
Let $p_1, p_2 \in (1,\infty)$ and $(X_1, X_2, X_3)$ be a compatible triple of function lattices.
For $j_1, j_2 \in \{1,2\}$ we have the paraproduct estimate
$$
\|A_{j_1, j_2}(b,f)\|_{L^{p_1}_{x_1} L^{p_2}_{x_2}(X_3)} \lesssim \|b\|_{\BMO_{\calD, \textup{prod}}(X_1)} \|f\|_{L^{p_1}_{x_1} L^{p_2}_{x_2}(X_2)}.
$$
\end{prop}
\begin{proof}
We will show the estimate for $A_{1, 2}(b,f)$ -- the other cases are very similar.
Let $g \in L^{p_1'}_{x_1} L^{p_2'}_{x_2}(X_3')$ and estimate
\begin{align*}
|\langle A_{1,2}&(b,f), g \rangle|\le \sum_{R = I_1 \times I_2} \Big\{ |\langle b, h_R\rangle|, \Big| \Big \langle f, h_{I_1} \otimes \frac{1_{I_2}}{|I_2|} \Big\rangle \Big|
|\langle g, h_{I_1}h_{I_1} \otimes h_{I_2} \rangle| \Big\} \\
&\lesssim \|b\|_{\BMO_{\calD, \textup{prod}}(X_1)} \Big\| \Big( \sum_{R = I_1 \times I_2} \Big| \Big \langle f, h_{I_1} \otimes \frac{1_{I_2}}{|I_2|} \Big\rangle \Big|^2
|\langle g, h_{I_1}h_{I_1} \otimes h_{I_2} \rangle|^2 \frac{1_R}{|R|} \Big)^{1/2}\Big\|_{L^1(X_1')}.
\end{align*}
Here we used Lemma \ref{lem:H1BMO}. We can now further dominate this by
$$
\| S_{\calD^1, M_{\calD^2}} f\|_{L^{p_1}_{x_1} L^{p_2}_{x_2}(X_2)} \| S_{\calD^2, M_{\calD^1}} g\|_{L^{p_1'}_{x_1} L^{p_2'}_{x_2}(X_3')},
$$
and the proof is finished using Lemma \ref{lem:standardEst}.
\end{proof}

\begin{rem}
In the one-parameter case, paraproduct estimates are known to hold in all $\UMD$ spaces (i.e., without imposing any additional function lattice structure), but the proofs become slightly more demanding under such minimal assumptions. For background and a modern proof of such estimates see e.g. \cite{HH}.
\end{rem}

We are now ready to give the proof of the bi-parameter analogue of \eqref{eq:vecval1par}.
The proof follows the modern paradigm of dyadic representation theorems of singular integrals combined with
the paraproduct decompositions of commutators.

\begin{rem}
A vector-valued estimate for $[b,T]$, where $T$ is a bi-parameter CZO \cite{Ma1}, could extremely likely also be proved with an adaptation of the strategy below.
\end{rem}

\begin{thm}\label{thm:vecval2par}
Let $p_1, p_2 \in (1,\infty)$ and $(X_1, X_2, X_3)$ be a compatible triple of function lattices. Let $T_{i}$ be a CZO in $\R^{d_i}$, $i=1,2$.
Then we have
$$
\|[T_1, [b, T_2]] f\|_{L^{p_1}_{x_1} L^{p_2}_{x_2}(X_3)} \lesssim \|b\|_{\BMO_{\textup{prod}}(X_1)} \|f\|_{L^{p_1}_{x_1} L^{p_2}_{x_2}(X_2)}.
$$
\end{thm}

\begin{proof}
The proof is a vector-valued adaptation of the recent commutator decompositions \cite{LMV:Bloom, LMV:BloomProdBMO, LMV:BilinBipar}. In particular, see \cite{LMV:BloomProdBMO} for a two-weight
version of this bound in the scalar-valued setting. To begin the proof, we recall that it is enough to fix lattices $\calD^i$ and
to estimate  $[U_1, [b, U_2]]f$, where $U_1$ and $U_2$ are DMOs -- dyadic model operators -- appearing in the representation theorem \cite{Hy, Hy2}.
That is, we have that $U_i \in \{S_i, \pi_i\}$, where $S_i$ is a so-called dyadic shift in $\R^{d_i}$ and $\pi_i$ is a dyadic paraproduct in $\R^{d_i}$. We will recall the
definitions during the proof. It is enough to obtain a polynomial dependence on the complexity of the shift.

We will only explicitly show the case $[S_1, [b, \pi_2]]f$ (see \cite{LMV:BloomProdBMO} for the decompositions used in the other cases). The dyadic shift $S_1$ has the form
$$
S_1 f = \sum_{ \substack{K_1 \in \calD^1 \\  I_1^{(i_1)} = J_1^{(j_1)} = K_1}} a_{K_1, I_1, J_1}  h_{J_1} \otimes \langle f, h_{I_1} \rangle_1,
$$
where $i_1,j_1  \ge 0$ and $|a_{K_1, I_1, J_1}| \le |I_1|^{1/2} |J_1|^{1/2} |K_1|^{-1}$. The paraproduct $\pi_2$ has the form
$$
\pi_2 f = \sum_{K_2 \in \calD^2} a_{K_2} \langle f \rangle_{K_2,2} \otimes h_{K_2},
$$
where $\|(a_{K_2})\|_{\BMO(\calD^2)} = \sup_{P_2 \in \calD^2} \Big( \frac{1}{|P_2|} \sum_{K_2 \subset P_2} |a_{K_2}|^2\Big)^{1/2} \le 1$.
We now write out
$$
[S_1, [b, \pi_2]]f = S_1(b \pi_2 f)  - b \pi_2 S_1 f -  S_1 \pi_2 (bf) + \pi_2 (b S_1 f).
$$
We identify the products 
$$
b \pi_2 f,\, b \pi_2 S_1 f,\, bf,\, b S_1 f
$$
that we want to expand using paraproducts. Of course, it is possible to expand these in the one-parameter sense (in either parameter) or in the bi-parameter sense.
What is actually useful is determined by the form of the operators $S_1$ and $\pi_2$, and looking at where the cancellative Haar functions appear. This leads us to e.g.
expand
\begin{align*}
b \pi_2 f &= \sum_{k_1, k_2 \in \{1,2, 3\}} A_{k_1, k_2}(b, \pi_2 f)\\
& = \sum_{k_1, k_2 \in \{1,2\}} A_{k_1, k_2}(b, \pi_2 f) + 
\sum_{ \substack{ (k_1, k_2) \ne (3,3) \\ 3 \in \{k_1, k_2\}}}  A_{k_1, k_2}(b, \pi_2 f) + A_{3, 3}(b, \pi_2 f),
\end{align*}
where the idea of the latter grouping is that in the case $k_1, k_2 \in \{1,2\}$ we can directly use Proposition \ref{prop:ParaLattice},
while the remaining terms cannot be handled alone. They need to be grouped with terms coming from the other product expansions, which are as follows:
$$
b \pi_2 S_1 f = \sum_{k_1, k_2 \in \{1,2\}} A_{k_1, k_2}(b, \pi_2 S_1 f) + 
\sum_{ \substack{ (k_1, k_2) \ne (3,3) \\ 3 \in \{k_1, k_2\}}}  A_{k_1, k_2}(b, \pi_2 S_1 f) + A_{3, 3}(b, \pi_2 S_1 f),
$$
$$
bf = \sum_{k_1 = 1}^2 A_{k_1}^1(b,f) + A_{3}^1(b,f)
$$
and
$$
b S_1 f = \sum_{k_1 = 1}^2 A_{k_1}^1(b,S_1f) + A_{3}^1(b,S_1 f).
$$
None of the one-parameter paraproducts are bounded as they are. However, when $k_1 \ne 3$, they will be paired with \textbf{one} other term
coming from the other product expansions. In this sense, they are similar to the terms of the form $A_{k_1, k_2}(b, F)$,  $(k_1, k_2) \ne (3,3)$, $3 \in \{k_1, k_2\}$.
Finally, all the four terms $A_{3, 3}(b, \pi_2 f)$, $A_{3, 3}(b, \pi_2 S_1 f)$, $A_{3}^1(b, f)$ and $A_{3}^1(b,S_1 f)$ will be grouped together.

The above explanation leads us to the long expansion
\begin{align*}
[S_1, [b, \pi_2]]f  &=  \sum_{k_1, k_2 \in \{1,2\}} S_1(A_{k_1, k_2}(b, \pi_2 f)) - \sum_{k_1, k_2 \in \{1,2\}} A_{k_1, k_2}(b, \pi_2 S_1 f) \\
&+ \Big\{ \sum_{ \substack{ (k_1, k_2) \ne (3,3) \\ 3 \in \{k_1, k_2\}}}  S_1(A_{k_1, k_2}(b, \pi_2 f)) + \sum_{k_1=1}^2 \pi_2(A_{k_1}^1(b, S_1 f)) \\
& \qquad - \sum_{ \substack{ (k_1, k_2) \ne (3,3) \\ 3 \in \{k_1, k_2\}}} A_{k_1, k_2} (b, \pi_2 S_1 f) - \sum_{k_1=1}^2 S_1 \pi_2(A_{k_1}^1(b,f)) \Big\} \\
&+ \big\{ [S_1(A_{3, 3}(b, \pi_2 f)) -S_1\pi_2(A_{3}^1(b, f))] + [\pi_2(A_{3}^1(b, S_1 f)) - A_{3, 3}(b, \pi_2 S_1 f)] \big\}.
\end{align*}
Using Proposition \ref{prop:ParaLattice} and the fact (see \cite{HH}) that $S_1$ and $\pi_2$ are bounded in $L^{s_1}_{x_1} L^{s_2}_{x_2}(X)$ for all $s_1, s_2 \in (1,\infty)$ and for all $\UMD$ function lattices $X$, we have that the first two terms in the expansion above automatically satisfy the desired estimate.
The big group in the middle consists of the various terms, where exactly one of the paraproduct indices is $3$ or we have
a one-parameter paraproduct with an indice $k_1 \ne 3$. They will not all be handled together, but rather by always pairing two suitable ones together (notice that
there are $6$ terms with a plus sign and $6$ terms with a minus sign). It is not straightforward to explain which terms should be paired together -- it can be seen by analysing their forms. We get to them later.

We start our work with the last term, which is the combination of the four terms, where all the paraproduct indices are $3$. 
First, we have by the definition of $A_{3,3}$ and $\pi_2$ that
\begin{align*}
A_{3,3}(b, \pi_2 f) &= \sum_{L_1, L_2} \langle b \rangle_{L_1 \times L_2} \langle \pi_2 f, h_{L_1} \otimes h_{L_2} \rangle h_{L_1} \otimes h_{L_2} \\
&= \sum_{L_1, L_2} \langle b \rangle_{L_1 \times L_2} \Big\langle \sum_{K_2} a_{K_2} \langle f \rangle_{K_2,2} \otimes h_{K_2} , h_{L_1} \otimes h_{L_2} \Big \rangle h_{L_1} \otimes h_{L_2} \\
&= \sum_{L_1, K_2} a_{K_2} \langle b \rangle_{L_1 \times K_2} \Big \langle f, h_{L_1} \otimes \frac{1_{K_2}}{|K_2|} \Big\rangle h_{L_1} \otimes h_{K_2}.
\end{align*}
Now, using the definition of $S_1$ we calculate
\begin{align*}
S_1(&A_{3, 3}(b, \pi_2 f)) = \sum_{ \substack{K_1 \\  I_1^{(i_1)} = J_1^{(j_1)} = K_1}} a_{K_1, I_1, J_1}  h_{J_1} \otimes \langle A_{3,3}(b, \pi_2 f), h_{I_1} \rangle_1 \\
&= \sum_{ \substack{K_1 \\  I_1^{(i_1)} = J_1^{(j_1)} = K_1}} a_{K_1, I_1, J_1}  h_{J_1} \otimes \Big \langle
\sum_{L_1, K_2} a_{K_2} \langle b \rangle_{L_1 \times K_2} \Big \langle f, h_{L_1} \otimes \frac{1_{K_2}}{|K_2|} \Big\rangle h_{L_1} \otimes h_{K_2}, h_{I_1} \Big\rangle_1 \\
&= \sum_{K_2} a_{K_2} \sum_{ \substack{K_1 \\  I_1^{(i_1)} = J_1^{(j_1)} = K_1}} a_{K_1, I_1, J_1} 
 \langle b \rangle_{I_1 \times K_2} \Big \langle f, h_{I_1} \otimes \frac{1_{K_2}}{|K_2|} \Big\rangle h_{J_1} \otimes h_{K_2}.
\end{align*}
The next task is to perform a completely analogous calculation to get that
$$
S_1\pi_2(A_{3}^1(b, f)) =
\sum_{K_2} a_{K_2} \sum_{ \substack{K_1 \\  I_1^{(i_1)} = J_1^{(j_1)} = K_1}} a_{K_1, I_1, J_1} 
 \bla \langle b\rangle_{I_1,1}\langle f, h_{I_1}\rangle_1\bra_{K_2}  h_{J_1} \otimes h_{K_2}.
$$
We can now write
\begin{equation}\label{eq:eq2}
\begin{split}
S_1(&A_{3, 3}(b, \pi_2 f)) - S_1\pi_2(A_{3}^1(b, f)) \\
&= \sum_{K_2} a_{K_2} \sum_{ \substack{K_1 \\  I_1^{(i_1)} = J_1^{(j_1)} = K_1}} a_{K_1, I_1, J_1} \Big\langle \big[\langle b\rangle_{I_1\times K_2}-\langle b\rangle_{I_1,1}\big]\langle f, h_{I_1}\rangle_1 \Big\rangle_{K_2}
   h_{J_1} \otimes h_{K_2}.
\end{split}
\end{equation}
Here the term $\big\langle \big[\langle b\rangle_{I_1\times K_2}-\langle b\rangle_{I_1,1}\big]\langle f, h_{I_1}\rangle_1 \big\rangle_{K_2}$ appears, which we further expand by applying the martingale difference expansion $g1_{K_2} = \sum_{I_2 \subset K_2} \langle g, h_{I_2} \rangle h_{I_2} + \langle g \rangle_{K_2}1_{K_2}$ to the function $\langle b\rangle_{I_1,1}$. This gives that
\begin{align*}
\Big\langle \big[\langle b\rangle_{I_1\times K_2}-\langle b\rangle_{I_1,1}\big]\langle f, h_{I_1}\rangle_1 \Big\rangle_{K_2} 
&= - \Big\langle \Big[ \sum_{I_2 \subset K_2} \langle \langle b, h_{I_2} \rangle_2 \rangle_{I_1} h_{I_2} \Big] \langle f, h_{I_1}\rangle_1 \Big\rangle_{K_2} \\
&= -\frac{1}{|K_2|} \sum_{I_2 \subset K_2} \langle \langle b, h_{I_2} \rangle_2 \rangle_{I_1}  \langle f, h_{I_1} \otimes h_{I_2} \rangle.
\end{align*}
Plugging this into \eqref{eq:eq2} gives the final
formula
\begin{align*}
S_1(&A_{3, 3}(b, \pi_2 f)) - S_1\pi_2(A_{3}^1(b, f)) \\
&= -\sum_{K_2} \frac{a_{K_2}}{|K_2|}
  \sum_{ \substack{K_1  \\  I_1^{(i_1)} = J_1^{(j_1)} = K_1}} \sum_{I_2 \subset K_2}  a_{K_1, I_1, J_1} 
  \langle \langle b, h_{I_2} \rangle_2 \rangle_{I_1}  \langle f, h_{I_1} \otimes h_{I_2} \rangle h_{J_1} \otimes h_{K_2}.
\end{align*}

This is clear progress as there is now one cancellative Haar function paired with $b$. On the other hand, this is simply not enough,
since we need two cancellative Haar functions to exploit the product $\BMO$ assumption.
However, with very similar calculation as above we see that
\begin{align*}
\pi_2(&A_{3}^1(b, S_1 f)) - A_{3, 3}(b, \pi_2 S_1 f) \\
&= \sum_{K_2} \frac{a_{K_2}}{|K_2|}
  \sum_{ \substack{K_1  \\  I_1^{(i_1)} = J_1^{(j_1)} = K_1}} \sum_{I_2 \subset K_2}   a_{K_1, I_1, J_1} 
  \langle \langle b, h_{I_2} \rangle_2 \rangle_{J_1}  \langle f, h_{I_1} \otimes h_{I_2} \rangle h_{J_1} \otimes h_{K_2}.
\end{align*}
The only difference is that we have a plus sign in front and that $\langle \langle b, h_{I_2} \rangle_2 \rangle_{I_1}$ is replaced by $\langle \langle b, h_{I_2} \rangle_2 \rangle_{J_1}$. We have arrived at the combined formula
\begin{align*}
 &[S_1(A_{3, 3}(b, \pi_2 f)) -S_1\pi_2(A_{3}^1(b, f))] + [\pi_2(A_{3}^1(b, S_1 f)) - A_{3, 3}(b, \pi_2 S_1 f)]  \\
 &=  \sum_{K_2} \frac{a_{K_2}}{|K_2|}
  \sum_{ \substack{K_1  \\  I_1^{(i_1)} = J_1^{(j_1)} = K_1}} \sum_{I_2 \subset K_2}   a_{K_1, I_1, J_1} [ \langle \langle b, h_{I_2} \rangle_2 \rangle_{J_1} - \langle \langle b, h_{I_2} \rangle_2 \rangle_{I_1} ] \\
&\hspace{8cm} \times 
\langle f, h_{I_1} \otimes h_{I_2} \rangle   h_{J_1} \otimes h_{K_2}.
\end{align*}
Next, we add and subtract $\langle \langle b, h_{I_2} \rangle_2 \rangle_{K_1}$ and e.g.
write
$$
\langle \langle b, h_{I_2} \rangle_2 \rangle_{J_1} - \langle \langle b, h_{I_2} \rangle_2 \rangle_{K_1}
= \sum_{l_1 = 1}^{j_1} \langle b, h_{J_1^{(l_1)}} \otimes h_{I_2} \rangle \langle h_{J_1^{(l_1)}} \rangle_{J_1}.
$$
It is now enough to fix $l_1 \in \{1, \ldots, j_1\}$ and consider, for $g \in L^{p_1'}_{x_1} L^{p_2'}_{x_2}(X_3')$, the following dualized term
\begin{align*}
  \sum_{K_1} \sum_{L_1^{(j_1 -l_1)} = K_1}&
   \sum_{I_2}   \Big\{ |\langle b, h_{L_1} \otimes h_{I_2} \rangle|, \\ &\sum_{ \substack{ I_1^{(i_1)} = K_1 \\ J_1^{(l_1)} = L_1}} |L_1|^{-1/2} |a_{K_1, I_1, J_1}|
|\langle f, h_{I_1} \otimes h_{I_2} \rangle| \sum_{K_2 \supset I_2} \frac{|a_{K_2}|}{|K_2|} | \langle g,  h_{J_1} \otimes h_{K_2}\rangle| \Big\}.
\end{align*}
Notice that we can estimate
\begin{align*}
\sum_{K_2 \supset I_2} \frac{|a_{K_2}|}{|K_2|}| \langle g,  h_{J_1} \otimes h_{K_2}\rangle|
&= \Big\langle \sum_{K_2 \supset I_2} |a_{K_2}| |\langle g,  h_{J_1} \otimes h_{K_2}\rangle| \frac{1_{K_2}}{|K_2|}  \Big\rangle_{I_2} \\
&\le |J_1|^{1/2} \Big\langle \sum_{K_2} |a_{K_2}| \bla |\langle g,   h_{K_2}\rangle_2|\bra_{J_1} \frac{1_{K_2}}{|K_2|}  \Big\rangle_{I_2},
\end{align*}
and use the normalisation of $|a_{K_1, I_1, J_1}|$ to estimate
\begin{align*}
|L_1|^{-1/2} |a_{K_1, I_1, J_1}|
|\langle f, h_{I_1} \otimes h_{I_2} \rangle|&\le |L_1|^{-1/2} \frac{|I_1|^{1/2}|J_1|^{1/2}}{|K_1|}
|\bla \langle \Delta_{K_1, i_1} f, h_{I_2} \rangle_2, h_{I_1}\bra|  \\
&\le  |L_1|^{-1/2} \frac{|J_1|^{1/2}}{|K_1|}\int_{I_1}| \langle \Delta_{K_1, i_1} f, h_{I_2} \rangle_2|.
\end{align*}
Now, by summing over $I_1, J_1$ we have reduced to bounding
\begin{align*}
  \sum_{K_1} \sum_{L_1^{(j_1 -l_1)} = K_1}&
   \sum_{I_2}   \Big\{ |\langle b, h_{L_1} \otimes h_{I_2} \rangle|, \\ & |L_1|^{1/2}
   \bla |\langle \Delta_{K_1, i_1} f, h_{I_2} \rangle_2| \bra_{K_1}
   \Big\langle \sum_{K_2} |a_{K_2}| \bla |\langle g, h_{K_2} \rangle_2| \bra_{L_1} \frac{1_{K_2}}{|K_2|}  \Big\rangle_{I_2}
\Big\}.
\end{align*}
We use Lemma \ref{lem:H1BMO} to dominate this by $\|b\|_{\BMO_{\calD, \textup{prod}}(X_1)} \le \|b\|_{\BMO_{\textup{prod}}(X_1)}$ multiplied with
\begin{align*}
\Big\| \Big( \sum_{K_1} \sum_{I_2} \big[ M_{\calD^1} \langle \Delta_{K_1, i_1} f, h_{I_2} \rangle_2 \big]^2 \otimes
\frac{1_{I_2}}{|I_2|} \Big)^{1/2} M_{\calD} \Big(\sum_{K_2} |a_{K_2}| |\langle g, h_{K_2} \rangle_2| \otimes \frac{1_{K_2}}{|K_2|} \Big)  \Big\|_{L^1(X_1')}.
\end{align*}
After using the estimate $|e_2e_3'|_{X_1'} \le |e_2|_{X_2} |e_3'|_{X_3'}$ this term decouples naturally by H\"older's inequality. After this
the proof can be ended using Lemma \ref{lem:standardEst} and the $L^{p_1'}_{x_1} L^{p_2'}_{x_2}(X_3')$ boundedness of the paraproduct
$$
g \mapsto \sum_{K_2} |a_{K_2}| |\langle g, h_{K_2} \rangle_2| \otimes \frac{1_{K_2}}{|K_2|}.
$$

It only remains to estimate the terms in the big block, i.e., the terms
\begin{align*}
 \sum_{ \substack{ (k_1, k_2) \ne (3,3) \\ 3 \in \{k_1, k_2\}}}  S_1(A_{k_1, k_2}(b, \pi_2 f)) &+ \sum_{k_1=1}^2 \pi_2(A_{k_1}^1(b, S_1 f)) \\
 &- \sum_{ \substack{ (k_1, k_2) \ne (3,3) \\ 3 \in \{k_1, k_2\}}} A_{k_1, k_2} (b, \pi_2 S_1 f) - \sum_{k_1=1}^2 S_1 \pi_2(A_{k_1}^1(b,f)).
\end{align*}
Here we always have to pair two terms together, one with a plus and one with a minus sign, to induce cancellation to $b$.
For example, we will consider
$$
S_1(A_{3, 1}(b, \pi_2 f)) - A_{3, 1} (b, \pi_2 S_1 f).
$$
Using that $\pi_2 S_1 = S_1 \pi_2$ and the boundedness of $\pi_2$, it is enough to consider
\begin{align*}
&S_1(A_{3, 1}(b, f)) - A_{3, 1} (b, S_1 f) \\
&= \sum_{ \substack{K_1  \\  I_1^{(i_1)} = J_1^{(j_1)} = K_1}} \sum_{K_2} a_{K_1, I_1, J_1}
\big( \bla \langle b, h_{K_2}\rangle_2 \bra_{I_1} - \bla \langle b, h_{K_2} \rangle_2 \bra_{J_1} \big)
\langle f, h_{I_1} \otimes h_{K_2} \rangle h_{J_1} \otimes h_{K_2} h_{K_2}.
\end{align*}
This can be bounded like the previous term, except it is simpler. This ends our proof.
\end{proof}

We note that the upper bound related to the lower bound of Proposition \ref{prop:RectBMO} -- that is, the upper bound related to the diagonal $p_1 = q_1$ and $p_2 = q_2$ of Theorem \ref{thm:main} -- is just the scalar-valued version of the above result.
However, this scalar-valued result is very well-known.
For general CZOs like here, the non-mixed norm version
appears in \cite{DaO}. The scalar-valued mixed-norm version could also directly be proved e.g. by extrapolating the estimate in \cite{LMV:BloomProdBMO}, which is a two-weight version of \cite{DaO}.

Lastly, even if the bi-commutators $[T_1, [b, T_2]]$ are the main topic of this paper, we record how the vector-valued estimate for the bi-commutator implies
an estimate for a commutator with three CZOs. Establishing a dotted version of the space $L^{r_3}_{x_3}$ would require some modifications as in the bi-commutator case.
\begin{prop}
Let $p_1, p_2, p_3, q_3 \in (1,\infty)$, $p_3 > q_3$ and define $r_3 \in (1,\infty)$ via
$$
\frac{1}{r_3} = \frac{1}{p_3} - \frac{1}{p_3}.
$$
Let $T_{i}$ be a CZO in $\R^{d_i}$, $i = 1,2, 3$. Then we have
$$
\|[T_1, [T_2, [b, T_3]]]f\|_{L^{p_1}_{x_1}L^{p_2}_{x_2}L^{q_3}_{x_3}} \lesssim \|b\|_{\BMO_{\textup{prod}, x_1, x_2}(L^{r_3}_{x_3})} \|f\|_{L^{p_1}_{x_1}L^{p_2}_{x_2}L^{p_3}_{x_3}}.
$$
\end{prop}
It is an interesting, but very complicated, problem to extend all of our results to tri-commutators, or even higher order commutators.

\end{document}